\documentclass[reqno]{amsart}
\usepackage[utf8]{inputenc}
\usepackage{amsmath}
\usepackage[shortlabels]{enumitem}
\usepackage {amssymb}
\usepackage{amsfonts}
\usepackage{amsthm}
\usepackage{color}
\usepackage{natbib}
\usepackage{graphicx}
\usepackage{mathtools}
\usepackage{MnSymbol}
\usepackage{hyperref}
\usepackage{algorithm}
\usepackage{caption}
\usepackage{subcaption}
\usepackage{float}
\usepackage[noend]{algpseudocode}
\hypersetup{colorlinks,citecolor=blue,urlcolor=blue,linkcolor=blue}

\newtheorem{theorem} {Theorem} 
\newtheorem{lemma}{Lemma}

\newtheorem{proposition} {Proposition} 
\newtheorem{corol} {Corollary} 
\newtheorem*{theorem-non}{Theorem}

\theoremstyle{remark}

\newtheoremstyle{named}{}{}{\itshape}{}{}{.}{.5em}{#1 #3}
\theoremstyle{named}

\newcommand{\abs}[1]{\left\lvert#1\right\rvert}
\newcommand{\Norm}[1]{\left\lVert#1\right\rVert}

 \newcommand\omicron{o}

\DeclarePairedDelimiter{\norm}{\lVert}{\rVert}

\author{Panagiotis Lolas}
\date{}

\title{Regularization in High-Dimensional Regression and Classification via Random Matrix Theory}

\begin{document}
\maketitle

\begin{abstract}
We study general singular value shrinkage estimators in high-dimensional regression and classification, when the number of features and the sample size both grow proportionally to infinity. We allow models with general covariance matrices that include a large class of data generating distributions. As far as the implications of our results are concerned, we find exact asymptotic formulas for both the training and test errors in regression models fitted by gradient descent, which provides theoretical insights for early stopping as a regularization method. In addition, we propose a numerical method based on the empirical spectra of covariance matrices for the optimal eigenvalue shrinkage classifier in linear discriminant analysis. Finally, we derive optimal estimators for the dense mean vectors of high-dimensional distributions. Throughout our analysis we rely on recent advances in random matrix theory and develop further results of independent mathematical interest.
\end{abstract}

\tableofcontents

\section{Introduction}

In recent years scientists in many different disciplines have access to very high-dimensional data sets. The abundance of data together with the significant increase in computing power has allowed them to perform analyses that would have been impossible some decades ago. On the other hand, it is known that classical statistical theory does not provide accurate explanations of the performance of the methods used when the data dimension increases to infinity (see, for example, \cite{bai_book}). As a result, serious effort has been put in the last years by statisticians to develop theory that better captures the relevant aspects in this setting and suggest better procedures for estimation, testing and prediction. Many of the methods that have  successfully been used in practice rely on incorporating prior knowledge to the estimation procedure. A common hypothesis in many applications is that each predictor variable has a small effect on the outcome, a statement that will be made precise later. A widely used method in this scenario is ridge regression and ridge-regularized discriminant analysis. Those methods were studied in detail in \cite{wager}, where the authors provide exact asymptotic formulas for the limiting out-of-sample predictive risk. In this paper, we vastly extend their results to more general shrinkage methods.

%\subsection{Our Goal}

For the case of regression the main assumption, which was also used in \cite{wager}, is that we observe i.i.d. data points $(x_i,y_i)\in\mathbb{R}^p\times\mathbb{R}$ from a linear model $y_i=w^\intercal x_i+\varepsilon_i$. We will use the notation $X=(x_1,\cdots,x_n)\in\mathbb{R}^{p\times n}.$ To model the assumption that each predictor variable has a small effect on the outcome the authors suggested the so-called Random Regression Coefficient Hypothesis (RRC), namely that the coefficent vector $w$ has i.i.d. coordinates with mean 0 and variance ${\alpha^2}{p^{-1}}$. We study estimates based on general shrinkage of the spectrum of the predictor matrix of the form $\hat{w}=\sum h(\lambda_i)u_iv_i^\intercal n^{-{1}/{2}}{y},$ where $\sum \sqrt{\lambda_i}u_iv_i^\intercal$ is the singular-value-decomposition of $n^{-{1}/{2}}{X}.$ These include as special cases the ridge regression estimate, but also the estimated coefficient vector after any step of descent-based optimization schemes. The extensions we provide here have important implications about early stopping in gradient descent-trained models and the interplay between the regularization parameter and the optimal stopping time in training. Early stopping has been used by practitioners in the deep learning community for a while (see \cite{bengio} for an explanation) as a regularization method. In many cases it has been observed that instead of fully minimizing the loss function in a regression or classification task, it is better for out-of-sample predictions to stop the decent optimization algorithm after fewer iterations. In our models this will become apparent for under-regularized ridge-regression with exact asymptotic formulas. For the case of simple linear regression with identity covariance matrix and no $l_2$ regularization parameter the learning curves have been studied in \cite{learning_curves}. For general covariance matrix asymptotic limits are derived in \cite{ali2018continuous}, but their formulas are rather intractable. Here we are going to get as a byproduct of our main theorems a more general analysis of the problem under an arbitrary covariance matrix of the predictor variables and the inclusion of an $l_2-$ penalty and provide explicit formulas for the limits for a more general class of estimators. We are also going to verify that no singular value shrinkage method can perform better than ridge regression with the Bayes optimal regularization parameter as a simple corollary of Theorem \ref{main_1}. A common assumption that allows us to use the asymptotic results of random matrix theory is that we are working in a regime with $p$-dimensional features and $n$ data points such that $p,n\rightarrow{\infty},{p}n^{-1}\rightarrow{\gamma>0}$. This assumption will be of great importance throughout this paper. In addition, we assume that the spectral distribution of the population covariance matrix $\Sigma$ of the predictor variables converges weakly to a deterministic limiting spectral measure $H$ supported on $[0,\infty).$ The model will be explained in detail in Section \ref{linear}.

For the case of classification using linear discriminant analysis in the dense setting that we consider in this paper, to the best of our knowledge, not much work has been done that departs from the ridge-regularized method proposed in \cite{wager}. The main assumption is that we have observed data points from two Gaussian distributions in $\mathbb{R}^p$ with $\mathcal{N}(\delta,\Sigma),\mathcal{N}(-\delta,\Sigma),$ where the coordinates of $\delta$ are i.i.d. with mean zero and variance ${\alpha^2}{p^{-1}}$. We then estimate $\delta,\hat{\Sigma}$ from the data and use them to classify new data points based on Bayes' rule. Since one main problem with the performance of discriminant analysis in high-dimensions is the noise in the estimation of the covariance matrix, one might attempt to perform eigenvalue shrinkage to the empirical covariance matrix before using it for a classification task in the same way as the authors do for covariance estimation in \cite{lw_linear}, \cite{lw_2012} and \cite{lw_direct}. Usually this is done by doing the spectral decomposition $\hat{\Sigma}=\sum_{i=1}^p \hat{\lambda_i}u_iu_i^\intercal$ and then using a function $h$ to shrink the eigenvalues and produce the estimate $h(\hat{\Sigma})=\sum_{i=1}^ph(\hat{\lambda_i})u_iu_i^\intercal$. Two questions arise naturally in discriminant analysis, if one uses such a method. The first one concerns the asymptotic prediction accuracy and the second one concerns whether it is possible to improve the performance by selecting a novel shrinkage method, since optimality of shrinkage functions has only been studied (as we will see below) for specific losses that do not include the classification error of discriminant analysis. Our theorems in Section \ref{DiscriminantAnalysis} will answer both questions later on. For this it will be necessary to prove convergence of certain trace functionals involving both the true and empirical covariance matrices of the predictor variables, a result of independent mathematical interest.

\subsection{Our Contributions}
We derive results of independent mathematical interest for the convergence of some trace functionals that include both the population and the sample covariance matrices in Proposition \ref{main_1}. This will be particularly useful for the proof of our main theorem (Theorem \ref{main_1}), which gives the asymptotic prediction risk for a large class of estimators for the coefficient vector $w$, but also has other important applications that we discuss. For instance, we explain in Corollary \ref{lw_opt_shrink} how it can be used to recover the optimal shrinkage function for covariance estimation in Frobenius norm derived heuristically in \cite{ledoit2011eigenvectors} and used in \cite{lw_2012}. Another interesting application is related to optimal shrinkage for estimation of means of high-dimensional distributions. This is done in Proposition \ref{prop:mean_delta}, where we suggest an empirical Bayes estimator with strong theoretical motivation. Moreover, Theorem \ref{main_1} implies the optimality of ridge regression with a suitable parameter among all singular value shrinkage-based methods. Finally, Theorem \ref{main_1} can be used to prove Proposition \ref{ridge_asympt}, which describes the test error evolution in gradient descent training. This gives us important insights about the early stopping regularization method used in neural networks. In particular, even in this very simplistic setting the regime we are studying explains some attributes of the overtraining phenomenon seen in much more complex models. 

For classification we prove Proposition \ref{Prop_LDA} as the main tool for our analysis. We use it to derive asymptotic formulas for the classification error of high-dimensional linear discriminant analysis under general nonlinear shrinkage functions for covariance estimation in Theorem \ref{main_discr}. After that, we study an optimization problem that gives the optimal shrinkage function for classification and explain how the optimum behaves for different values of signal-to-noise ratio. We also study a relaxation of the problem that can be solved analytically and show that a combination of the shrinkage for covariance estimation in \cite{lw_2012} and ridge-regularization is close to optimal in the sense that it provides the solution to this relaxation. This important observation can be used to approximate the optimal shrinkage at the same computational cost as tuning a ridge parameter for LDA. Before the development of the tools that we present here this analysis was impossible for general nonlinear shrinkage functions. Proposition \ref{Prop_LDA} that gives exact asymptotics for $\lim_{p\rightarrow{\infty}}p^{-1}tr\left(\Sigma h(\hat{\Sigma})\Sigma h(\hat{\Sigma})\right)$ for any continuous function $h,$ where $\Sigma$ is the population covariance matrix and $\hat{\Sigma}$ is the sample covariance matrix was particularly useful for this purpose.

\subsection{Related Work}

In the same direction as \cite{wager}, the authors in \cite{hastie2019surprises} study least-norm high-dimensional regression in a setting similar to ours. They also consider random kernel features interpolators and study the asymptotic risk. For that purpose the predictor variables are generated as $x_i=\phi(Wz_i),z_i\sim \mathcal{N}(0,I_d),$ where $W\in\mathbb{R}^{p\times d}$ is a matrix with i.i.d. entries with distribution $\mathcal{N}(0,d^{-1})$ and $d$ is an integer that grows proportionally to $p,n.$ This corresponds to the linearization of a two-layer neural network with random weights in the first layer. Interesting behaviour arises with this model for minimum-norm regression due to the double descent shape of the risk curve. Random features regression is also studied in \cite{montanari_mei} for learning an unknown function over the $d$-dimensional sphere via ridge regression.

As mentioned earlier, incorporating prior knowledge is usually extremely important for extending statistical methods to high dimensions. A different set of hypotheses are related to sparsity, according to which only a relatively small number of parameters in a regression or classification task are nonzero (\cite{hastie_tibsh_wain}). One of the first papers that showed that even in the extremely high-dimensional regime with $p>>n$ it is possible, under certain assumptions, to recover (with high probability) the coefficient vector was \cite{candes_tao}. Other papers in the vast literature in this direction are, for example, \cite{bayati}, where the authors derive the asymptotic risk for LASSO, or  \cite{montanari_donoho}, where the authors study the asymptotic variance of M-estimators. 

As far as high-dimensional classification is concerned, in \cite{bickel2004some} it is explained that classification using the full sample covariance matrix gives poor performance and a Naive Bayes classifier is suggested instead. In \cite{fan2008high} the authors argue that such independence rules may not be enough and suggest that feature selection rules are necessary. In \cite{bickel2008covariance} the closely related problem of covariance estimation is studied and thresholding-based methods are examined. For high-dimensional logistic regression, which is another commonly used classification method, high-dimensional phenomena are studied in \cite{candes2018phase},\cite{candes_sur}, \cite{candes_sur_1}. The authors consider extremely important questions such as the existence and asymptotic distribution of the Maximum-Likelihood Estimator for logistic regression when the number of features grows proportionally to the sample size. The asymptotic loss for a large class of binary classifiers (including logistic regression and maximum margin classifiers) in the high-dimensional regime considered here was studied in \cite{taheri2020sharp}.

Eigenvalue shrinkage and thresholding methods have been used in solving many problems. First of all, shrinkage methods for covariance estimation have been studied previously by many researchers. The authors in \cite{lw_linear} proposed a well-conditioned shrinkage estimator for the empirical covariance matrix based on a linear shrinkage method. In \cite{lw_2012} the authors extended their method to the study the optimal (with respect to some loss) rotation invariant nonlinear shrinkage estimator for a covariance matrix using results from random matrix theory and in \cite{lw_direct} kernel estimation was used to improve the numerical stability and speed of the optimal nonlinear shrinkage procedure. Finally, optimal shrinkage functions for a spiked covariance model for 26 different loss functions were obtained (for most of them analytically) in \cite{johnstone}. A closely related application in symmetric matrix denoising via singular value thresholding can be found in \cite{donoho_smd} and \cite{gavish2014optimal}.

\subsection{Organization of the Paper}
In Section \ref{prelim} we review some useful results from random matrix theory. Some of those results have been used in \cite{wager}, but here we will be required to develop more mathematical tools to derive our results. This is done, for example, in Proposition \ref{gen_lp}, which generalizes the results of \cite{ledoit2011eigenvectors}, or in Proposition \ref{modif_lda}. In  Section \ref{linear} we derive a general formula for the asymptotics of out-of-sample predictive risk in regularized linear regression. As mentioned earlier we will apply those results to understand early stopping and study the training and test error evolution. In Section \ref{DiscriminantAnalysis} we carry out a similar study for Discriminant Analysis. Furthermore, we explain how to numerically solve for the optimal shrinkage function. Finally, we summarize our most important results and propose some directions that might be useful to explore in the future in Section \ref{future}. Proofs of more technical results can be found in Section \ref{app}.

\bigskip
\textbf{Acknowledgements}

The authors are grateful to Lexing Ying, Emmanuel Candès and Nikolaos Ignatiadis for providing many useful comments 
that helped improve earlier versions of this paper.

\section{Results from Random Matrix Theory}\label{prelim}

We review some basic results about the asymptotics of the eigenvalue distributions of random matrices.

For a probability measure $\mu $ supported on the real line the Stieltjes transform is defined as $m_{\mu}(z)=\int (x-z)^{-1}{\mu(dx)}$ for $z\in\mathbb{C}$ away from the support of $\mu$.

The spectral distribution of a Hermitian $p\times p$ matrix $A$ with eigenvalues (in decreasing order) $\lambda_1(z), \cdots,\lambda_p(A)$ is defined as the measure $F_A(x)=p^{-1}\sum_{i=1}^p\delta_{\lambda_i(A)}$. This is, of course, a measure supported on the real line. For empirical covariance matrices the spectral disribution is characterized asymptotically by the Marcenko-Pastur theorem. In the general form presented below it can be found in \cite{silverstein1995empirical}.

\begin{theorem-non}
Let $Z\in\mathbb{R}^{p\times n}$ be a matrix with i.i.d. mean 0 and variance 1 entries and $\Sigma\in\mathbb{R}^{p\times p}$ a covariance matrix that is either deterministic or random but independent of $Z$ and $X=\Sigma^{{1}/{2}}Z.$ We assume that the spectral distribution of $\Sigma$  converges weakly to a deterministic probability measure $H$ supported on $[0,\infty)$. If $p,n\rightarrow{\infty}$ such that ${p}{n^{-1}}\rightarrow{\gamma>0}$, then the spectral distribution of the matrix $n^{-1}{XX^\intercal}$ converges weakly almost surely to a deterministic probability measure $F_{\gamma,H}$ with Stieltjes transform that satisfies $$m_{\gamma,H}(z)=\int \frac{dH(t)}{t(1-\gamma-\gamma z m_{\gamma,H}(z))-z}.$$
\end{theorem-non}

We will often omit the subscripts and just write $m$ for the Stieltjes transform, since it will be clear from our setting what the subscripts correspond to. In many cases it is useful to work with the companion Stieltjes transform which is defined as $\underline{m}_{\gamma,H}(z)=-{(1-\gamma)}{z^{-1}}+\gamma m_{\gamma,H}(z)$ and corresponds to the Stieltjes transform of the limiting spectral distribution of $n^{-1}{X^\intercal X}$. For any $x\in \mathbb{R}-\{0\}$ \cite{silver_choi} proved that the limit $\displaystyle \lim_{\epsilon\downarrow 0}\underline{m}_{\gamma,H}(x+i\epsilon)= \Tilde{m}_{\gamma,H}(x)=f(x)+ig(x)$ exists. They also proved that $F$ has a continuous density away from $0$ given by $F'={Im(\Tilde{m})}{(\gamma\pi)^{-1}}={g}(\gamma\pi)^{-1}$.

This theorem has far-reaching implications in statistics. The references \cite{johnstone2006high} and \cite{paul2014random} provide interesting reviews. 

For the case $H=\delta_1$ the equation for the Stieltjes transform is quadratic and can be solved explicitly.  For other measures $H$ it is possible to solve numerically for the limiting spectral density, as explained in detail in \cite{dobriban2015efficient}. The author also provides software implementations of the methods.

A generalization of the Marcenko-Pastur equation is given in \cite{ledoit2011eigenvectors} and we mention an implication of their main result below. For the rest of this Section we are going to use the notation $S_n=n^{-1}{XX^\intercal}.$

\begin{theorem-non}\label{ledoit_peche}
With the assumptions above, if the entries of $Z$ have uniformly bounded $12$-th  moments, $H$ has  support contained in a compact interval $[h_1,h_2]$ and all the eigenvalues of $\Sigma$ eventually lie in $[h_1,h_2]$, then for any bounded continuous function $u$ and any $z\in\mathbb{C^+}$ we have \begin{equation}\label{lp}\frac{1}{p}tr\left(u(\Sigma) \left(S_n-z\right)^{-1}\right)\xrightarrow{a.s.}\int \frac{u(t)dH(t)}{t(1-\gamma-\gamma z m_{\gamma,H}(z))-z}.\end{equation} For $g(t)=1$ we recover the Marcenko-Pastur equation.
\end{theorem-non}

The authors used this result to study the overlap of the eigenvectors of $n^{-1}{XX^\intercal}$ with their population counterparts. Another significant application of this result in statistics was the nonlinear shrinkage method for covariance estimation mentioned earlier. Furthermore, it was used by \cite{wager} in their analysis of the asymptotics of ridge regression and ridge-regularized linear discriminant analysis.

An extension that we are going to use is proved below. This is an important result of independent mathematical interest. For our purposes it is going to be the main tool in providing limits for regularized high-dimensional linear regression risk. \textit{To avoid technical complications arising from the fact that the density of the Marcenko-Pastur can be unbounded when $\gamma=1,$ we are going to assume for the rest of the paper that $\gamma\neq 1.$}

\begin{proposition}\label{gen_lp}

With the same assumptions as above for any bounded continuous function $h$ on $[0,\infty)$ we have $$\frac{1}{p}tr\left(\Sigma h\left(S_n\right)\right)\xrightarrow{a.s.}M_{\gamma,H}(h),$$ where $$M_{\gamma,H}(h)=\int h(x) \frac{g(x)}{\pi\gamma x(f(x)^2+g(x)^2)}dx+\frac{h(0)}{\gamma\underline{m}(0)}I_{\gamma>1}.$$

\end{proposition}

\begin{proof}

We start the proof assuming $\gamma< 1$, but in the end we are going to show how to take into consideration the different behaviour that occurs for $\gamma>1$.

First of all, by a simple density argument, namely that polynomials on a compact set are dense in the uniform convergence topology in the space of continuous functions, it is enough to prove the result for polynomial functions $h$. We will prove it for holomorphic functions in general using Cauchy's integral formula. The fact that we can restrict our attention to a compact subset of $[0,\infty)$ follows from the inequality  $$\limsup{\Norm{S_n}}\leq \limsup{\Norm{\Sigma} {\Norm{\frac{ZZ^\intercal}{n}}}}\leq h_2(1+\sqrt{\gamma})^2.$$ Here we used the fact that $\limsup{\Norm{n^{-1}{ZZ^\intercal}}}=(1+\sqrt{\gamma})^2$ by \cite{bai1998no}.

We observe that $1+z\underline{m}(z)=\gamma(1+zm(z))$.

By \eqref{lp} we have
\begin{equation}\label{contour_formula}\begin{split}
\frac{1}{p}tr\left(\Sigma\left(S_n-z\right)^{-1}\right)\xrightarrow{a.s.}\int \frac{t}{t(1-\gamma-\gamma z m(z))-z}d H(t)\\
=\frac{1}{1-\gamma-\gamma z m(z)}\int (1+\frac{z}{t(1-\gamma-\gamma z m(z))-z})d H(t)\\
=\frac{1}{1-\gamma-\gamma z m(z)}(1+z m(z))=\frac{1}{\gamma}\frac{1+z\underline{m}(z)}{1-1-z\underline{m}(z)}=-\frac{1+z\underline{m}(z)}{\gamma z \underline{m}(z)}.
\end{split}
\end{equation}

Let $\Gamma$ be a simple closed curve that encloses counterclockwise the support of $F_{\gamma,H}$.

We write \begin{equation}\label{ocnt}\begin{split}
\frac{1}{p}tr\left(\Sigma h\left(S_n\right)\right)
\end{split}=-\frac{1}{2\pi i}\oint_{\Gamma}h(w)\frac{1}{p}tr\left(\Sigma \left(S_n-w\right)^{-1}\right)dw
\end{equation}

We clearly have that $\displaystyle \Norm{\left(S_n-w\right)^{-1}}$ is uniformly bounded almost surely for $w\in \Gamma$. In addition, $$\abs{\frac{d}{dw}\frac{1}{p}tr\left(\Sigma \left(S_n-w\right)^{-1}\right)}\leq \Norm{\Sigma}\norm{\left(S_n-w\right)^{-2}}.$$
 As a consequence, the sequence of functions $p^{-1}tr\left(\Sigma \left(S_n-w\right)^{-1}\right)$ almost surely contains functions that are uniformly Lipschitz on $\Gamma$. We conclude that almost surely the convergence in \eqref{contour_formula} is uniform for $w\in\Gamma$. By the bounded convergence theorem and combining \eqref{contour_formula} and \eqref{ocnt} it follows that \begin{equation}\begin{split}
\frac{1}{p}tr\left(\Sigma h\left(S_n\right)\right)\xrightarrow{a.s.}\frac{1}{2\pi i}\oint_{\Gamma} h(w)\frac{1+w\underline{m}(w)}{\gamma w\underline{m}(w)} dw
\end{split}
\end{equation}

We need to convert the contour integral to a real integral. To do this, we follow the idea of \cite{bai2008clt} and consider the curve $\Gamma$ to approximate an interval on the real line in both directions. Then, $$\lim_{\epsilon\rightarrow{0}^+}\frac{1+(x-i\epsilon)\underline{m}(x-i\epsilon)}{\gamma (x-i\epsilon)\underline{m}(x-i\epsilon)}-\frac{1+(x+i\epsilon)\underline{m}(x+i\epsilon)}{\gamma (x+i\epsilon)\underline{m}(x+i\epsilon)}$$
$$=-2i\lim_{\epsilon\downarrow 0}Im\left(\frac{1+(x+i\epsilon)\underline{m}(x+i\epsilon)}{\gamma (x+i\epsilon)\underline{m}(x+i\epsilon)}\right)=2i\frac{g(x)}{\gamma x(f(x)^2+g(x)^2)}, $$
as a simple calculation shows.

To adjust the proof for the case $\gamma> 1,$ we see that the only difference is that $\lim_{z\rightarrow{0}}z\underline{m}(z)=0$, so that we need to substitute the curve $\Gamma$ with two other curves, one (call it $\Gamma_1$) contained in $\{z\in\mathbb{C}:Re(z)> 0\}$ enclosing counterclockwise the support of the limiting spectral distribution $F_{\gamma,H}$ in the positive real axis, and the other (call it $\Gamma_2$) being a small circle around $0$ with radius $r$. 

Then, for the first curve the argument is the same as above and shows convergence to the same term exactly. For $\Gamma_2$ we have \begin{equation}\begin{split}\lim_{r\downarrow{0}}\frac{1}{2\pi i}\oint_{\Gamma_2} h(w)\frac{1+w\underline{m}(w)}{\gamma w\underline{m}(w)}dw
=\lim_{r\downarrow{0}}\frac{1}{2\pi i}\oint_{\Gamma_2}\frac{h(w)}{\gamma w
\underline{m}(w)}dw\\
=\lim_{r\downarrow{0}}\frac{1}{2\pi}\int_{0}^{2\pi}\frac{h(r\exp(i\theta))}{\gamma \underline{m}({r\exp(i\theta)})}d\theta
=\frac{h(0)}{\gamma \underline{m}(0)}.\end{split}
\end{equation}

This completes the proof.

\end{proof}

\textbf{Remarks}

If $\Sigma=I$ we have $f(x)^2+g(x)^2=x^{-1}$ and the result is equivalent to the original Marcenko-Pastur theorem. In addition, if we consider functionals of the form $p^{-1}tr(f_1(\Sigma)f_2(S_n))$, it is straightforward to derive an analogous formula.

%\subsection{Applications of Proposition \ref{gen_lp}}

As an application of Proposition \ref{gen_lp} that illustrates how it can be useful, we observe that one can recover the optimal nonlinear shrinkage function of \cite{ledoit2011eigenvectors}. There it was motivated heuristically as an asymptotic equivalent to an oracle that is the optimal rotation invariant estimator of the covariance matrix. We present this here from a different perspective, since a similar method will be applied in the Discriminant Analysis case to find the optimal shrinkage function there (numerically, since the formulas will be much more complicated).

\begin{corol}\label{lw_opt_shrink}
If $\gamma\neq 1$, among the bounded continuous functions on an open set $U$ that contains the support of $F_{\gamma,H}$, the one that minimizes the asymptotic (mean squared) Frobenius loss $\displaystyle \lim_{p\rightarrow{\infty}}p^{-1}\Norm{\Sigma-h\left(S_n\right)}_{F}^2$ satisfies $$h(x)=\frac{1}{x (f(x)^2+g(x)^2)},x\in supp(F_{\gamma,H}),x>0.$$ If $x=0$ and $\gamma>1,$ the optimal function at $0$ evaluates to $$h(0)=\frac{1}{(\gamma-1)\underline{m}(0)}.$$
\end{corol}

\begin{proof}

We will use the well-known fact that $$dF_{\gamma,H}(x)=\frac{g(x)}{\pi \gamma}dx+\max(1-\gamma^{-1},0)\delta_0.$$
The limit of the Frobenius loss is $$\int t^2dH(t)+\int h(x)^2dF_{\gamma,H}(x)-2\int h(x)\frac{g(x)}{\pi\gamma x(f(x)^2+g(x)^2)}dx-2\frac{h(0)}{\gamma \underline{m}(0)}I_{\gamma>1}$$

$$=\int t^2dH(t)+\int \left(h(x)^2-2h(x)\frac{1}{x(f(x)^2+g(x)^2)}\right)\frac{g(x)}{\pi \gamma}dx$$ $$+\left((1-\gamma^{-1})h(0)^2-2\frac{h(0)}{\gamma \underline{m}(0)}\right)I_{\gamma>1}.$$

For $\gamma<1$, we see that the integrand is minimized for $\displaystyle h(x)=(x(f(x)^2+g(x)^2))^{-1},$ which is going to be continuous close to the support of $F_{\gamma,H}$. For $\gamma>1,$ this is still true away from 0, but at $0$ we need to minimize $$\left((1-\gamma^{-1})h(0)^2-2\frac{h(0)}{\gamma \underline{m}(0)}\right).$$ This happens for $h(0)=(\gamma-1)^{-1}\underline{m}(0)^{-1}$. The fact that we there exists such a continuous function $h$ is trivial (for example, usually constructed by partitions of unity as in \cite{munkres}). This completes our proof.

\end{proof}

%\bigskip

%\textbf{High-Dimensional Mean Estimation}

%The next application is related to estimation of means of high-dimensional Gaussian distributions. Assume that we observe $x_1,\cdots,x_n\sim \mathcal{N}(\delta, )$

%A second application we are going to consider comes from finance, namely from Markowitz portfolio theory. Suppose that an agent wants to invest in $p$ risky assets, with an average return vector $\mu$ and covariance matrix $\Sigma$. 

\section{High-Dimensional Linear Regression}\label{linear}

\textbf{Regression Model and Assumptions} 

\textit{Data-Generating Distribution}

Suppose that we observe $n$ predictor variables $x_1,\cdots,x_n\in\mathbb{R}^p$ which are i.i.d. and there exist i.i.d. $z_1,\cdots,z_n\in\mathbb{R}^p$ with entries of mean 0, variance 1 and uniformly bounded $12$-th moments such that $x_i=\Sigma^{{1}/{2}}z_i.$ In addition, we assume that $p,n\rightarrow{\infty}$ with ${p}n^{-1}\rightarrow{\gamma>0}$. As in the previous section, $\Sigma$ is the population covariance matrix, which has bounded norm and spectral distribution that converges weakly to a probability measure $H$ compactly supported on $[0,\infty)$. The target variables $y_1,\cdots,y_n$ are generated by a linear model as $y_i=w^\intercal x_i+\varepsilon_i.$ As in the previous Section we are going to use the notation $S_n=n^{-1}{XX^\intercal}.$

\bigskip

\textit{Weight and error distributions}

The error terms are i.i.d. with mean 0, variance 1 and uniformly bounded $(4+\eta)$-th moments (for some $\eta>0$), while the weight vector $w$ has independent components with mean zero and variance $p^{-1}{\alpha^2}.$ This is a way to describe the hypothesis discussed in the introduction that each predictor variable has a small effect on the outcome. We also assume that the normalized coordinates $\sqrt{p}w_i$ of the weight vector have uniformly bounded $(4+\eta)$-th moments. The same assumptions are made in \cite{wager}. \bigskip

\textit{Out-of-sample risk}

The out-of-sample prediction risk is defined as $\mathbb{E}[(y_0-\hat{w}^\intercal x_0)^2|X,y,w]=1+\norm{\Sigma^{{1}/{2}}(\hat{w}-w)}^2,$ where $(x_0,y_0)$ is generated independently from the training sample by the same model.
\bigskip

\textit{Estimators that we consider}

We write $X=(x_1,\cdots,x_n)\in \mathbb{R}^{p\times n},y=(y_1,\cdots,y_n)^\intercal.$  When we perform linear regression of $X$ on $y$ (assume for a moment that ${XX^\intercal}$ is invertible) we estimate the weights by $\hat{w}=({XX^\intercal})^{-1}{X y}.$

If $n^{-{1}/{2}}{X}=\sum_{i=1}^p\sqrt{\lambda_i}u_iv_i^\intercal$ is the singular value decomposition of ${n}^{-1/2}{X}$, we can rewrite $\hat{w}=\sum_{i=1}^p{(\sqrt{n\lambda_i})^{-1}}u_iv_i^\intercal {y}.$ Similarly, ridge regression with parameter $\lambda\geq 0$ gives $\hat{w}_{\lambda}=\sum_{i=1}^p{\sqrt{n^{-1}\lambda_i}}{(\lambda_i+\lambda)^{-1}}u_iv_i^\intercal {y}.$ Thus, ridge regression can be treated as a shrinkage method for the singular values of $X$ before performing a linear regression. More examples of this phenomenon will be studied later. 

Motivated by this we look at estimators of the form $$\hat{w}=\sum_{i=1}^ph({\lambda_i})u_iv_i^\intercal \frac{y}{\sqrt{n}},$$ where $h$ is a bounded continuous function in an open interval containing the support of $F_{\gamma,H}$.

Our first important result is below. The functions $f,g$ are defined as in Section \ref{prelim}. 

\begin{theorem}\label{main_1}
The out-of-sample prediction risk under the assumptions described above converges almost surely to $$1+\int [\alpha^2(\sqrt{x}h(x)-1)^2+\gamma h(x)^2]\frac{g(x)}{\gamma\pi x (f(x)^2+g(x)^2)}dx+\frac{\alpha^2+\gamma h(0)^2}{\gamma\underline{m}(0)}I_{\gamma >1}.$$
\end{theorem}

The proof is given in Section \ref{app}. The two most important steps at the proof are to first control the deviations of the quadratic form in the out-of-sample risk around the expectation using a consequence of the Marcinkiewicz-Zygmund inequality and then derive asymptotic formulas for the expectations of the terms that appear using Proposition \ref{gen_lp}.
\bigskip

\textbf{Remark} Observe that the function $g$ is nonnegative, as ${(\pi \gamma)^{-1}}g$ is the density of the generalized Marcenko-Pastur distribution $F_{\gamma,H}.$ The function $\alpha^2(\sqrt{x}h(x)-1)^2+\gamma h(x)^2$ is quadratic in $h$ and is minimized when $\displaystyle h(x)={\sqrt{x}}(x+{\gamma}{\alpha^{-2})^{-1}}.$ This shrinkage function corresponds to ridge regression with parameter $\lambda^*={\gamma}{\alpha^{-2}}.$ One intuitive explanation for why this should be true is the following: From the Bayesian perspective, if we impose a normal prior distribution on the coefficient vector $w$ and the residual vector $\varepsilon$, after observing the data points $(x_i,y_i)$ the posterior of $w$ is going to be normal. The posterior mean is exactly the ridge regression estimate with the parameter above. In the case of a general prior distribution on $w$ the optimal shrinkage function should not change asymptotically, since the limit arising in Theorem \ref{main_1} is universal. In particular, we never made assumptions about the strict shape of the distribution of the coefficients except for the first two moments, when deriving the asymptotic formula of the error. Choosing the optimal parameter $\lambda^*$ may not be a straightforward task when $p>n$ and there has been significant amount of work in the literature dedicated to this question and the closely related heritability problem. For instance, the reader can refer to \cite{janson2017eigenprism} and \cite{dicker2016maximum} and the references therein.% corresponds to ridge regression with the Bayes optimal regularization parameter $\lambda^*=\frac{\gamma}{\alpha^2}.$ \textit{As a conclusion, the optimal regularization method that uses eigenvalue shrinkage of the predictor matrix is ridge regression with the Bayes optimal parameter.} 

\subsection{Learning Curves in Ridge Regression}

In the introduction we argued that Theorem \ref{main_1} will allow us to study the learning curves for high-dimensional regression for any covariance matrix $\Sigma$ and any regularization parameter $\lambda$ trained by gradient descent, which we do here.

Before we present our result, we need to describe the dynamics of descent in linear regression. The weights $\hat{w}(t),$ which we initialize at $0$ are iteratively updated to minimize over $\beta\in\mathbb{R}^p$ the function $$\frac{\norm{y-X^\intercal \beta}}{2n}+\frac{\lambda\norm{\beta}^{2}}{2}.$$ Taking small steps against the gradient of this function with step size $dt$ gives $$\hat{w}(t+dt)=\hat{w}(t)-dt\left(\frac{X(X^\intercal \hat{w}-y)}{n}+\lambda\hat{w}\right).$$ Assuming $dt\rightarrow{0}$ we get an ordinary differential equation for $\hat{w}(t)$ described by $$\frac{d\hat{w}}{dt}=-\left(\frac{XX^\intercal}{n}+\lambda\right)\hat{w}+\frac{Xy}{n}\Rightarrow{\hat{w}(t)=\left(I-\exp{\left(-t\left(S_n+\lambda\right)\right)}\right)}\left(S_n+\lambda\right)^{-1}\frac{Xy}{n}.$$
Notice that $\lim_{t\rightarrow{\infty}}\hat{w}(t)$ recovers the usual weight for ridge regression. We conclude that Theorem \ref{main_1} provides asymptotic formulas for the out-of-sample prediction risk at all points in time during the training. In particular, early stopping provides regularization via singular value shrinkage of $n^{-1/2}{X}.$ Using the notation we introduced above, we can apply our results for $h(x)=(1-\exp{(-t(x+\lambda))}){\sqrt{x}}(x+\lambda)^{-1}.$

\begin{proposition}\label{ridge_asympt}
The prediction risk using gradient descent for time $t$ in ridge regression with regularization parameter $\lambda$ converges almost surely to $$1+\int [\alpha^2(\sqrt{x}h(x)-1)^2+\gamma h(x)^2]\frac{g(x)}{\gamma\pi x (f(x)^2+g(x)^2)}dx+\frac{\alpha^2}{\gamma \underline{m}(0)}I_{\gamma>1},$$ where $h=h(x;t,\lambda)=(1-\exp{(-t(x+\lambda))}){\sqrt{x}}(x+\lambda)^{-1}.$
\end{proposition}

In Figure \ref{fig:verif_two_mas} verify experimentally that our theorem produces the right result for $p=500,n=1500,\alpha=1$ and $H=0.5(\delta_1+\delta_4).$ The regularization parameter was set to the optimal $\lambda={1}/{3}.$ In Figure \ref{fig:verif_auto_regress} we do the same for $p=1000,n=1500,\alpha=2$ an autoregressive covariance matrix with $\Sigma_{ij}=2^{-\abs{i-j}}.$ For this example we chose $\lambda=0,$ which corresponds to unregularized linear regression.
\begin{figure}[h]
\includegraphics[width=0.5\linewidth]{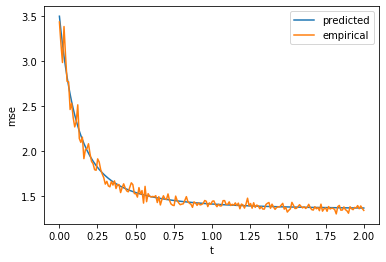}
\caption{Predicted vs Empirical out-of-sample MSE evolution for $p=500,n=1500,\alpha=1$ and $H=0.5{(\delta_1+\delta_4)}.$ Here the ridge parameter is chosen as the optimal $\lambda={1}/{3}.$}
\label{fig:verif_two_mas}
\end{figure}

\begin{figure}[h]

\includegraphics[width=0.5\linewidth]{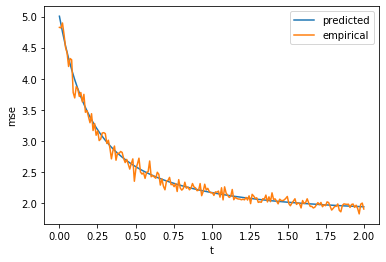}
\caption{Predicted vs Empirical out-of-sample MSE evolution for $p=1000,n=1500,\alpha=2,\lambda=0$ and $\Sigma_{ij}=2^{-\abs{i-j}}.$}
\label{fig:verif_auto_regress}
\end{figure}
\bigskip

Observe that the first term in the integral is decreasing in $t$ and roughly corresponds to a bias term, as we see from the proof. The second term in the integral is increasing in $t$ and behaves as a variance term. 

As a special case we recover a result of \cite{learning_curves} which we state in Corollary \ref{corol_advani}.

\begin{corol}\label{corol_advani}
If $\lambda=0$ (unregularized linear regression), for $\Sigma=I_p$, the evolution of the prediction risk during training is given by $$1+\int \alpha^2 \exp(-2tx)+\gamma\frac{(1-\exp{(-tx)})^2}{x}dF_{\gamma}(x).$$ Here $F_{\gamma}=F_{\gamma,\delta_1}$ is the measure on the real line that is absolutely continuous with respect to the Lebesgue measure for $\gamma\leq 1$ with density $$\frac{dF_{\gamma}(x)}{dx}=\frac{\sqrt{\left((1+\sqrt{\gamma}\right)^2-x)\left(x-(1-\sqrt{\gamma}\right)^2)}}{2\pi\gamma x},(1-\sqrt{\gamma})^2\leq x\leq (1+\sqrt{\gamma})^2.$$ If $\gamma>1,$ $F_{\gamma}$ has an additional mass $1-\gamma^{-1}$ at 0. 

\end{corol}

When $\gamma$ increases the effect of the second term will become more significant and in general will require stopping earlier in the training to achieve better performance. For $\alpha=0.5$ and three different values of $\gamma$ the asymptotic prediction risk is seen in the plot below in the null case $H=\delta_1$ and with $\lambda=0$. We observe that it is optimal to stop early, since going beyond a certain point results to overtraining, which increases the prediction risk. We also see that, as we explained, the higher $\gamma$ is, the earlier the overtraining phenomenon starts to impact the performance of the model. Furthermore, if $\alpha$ increases the optimal training time always increases, keeping the other parameters fixed. In other words, the smaller the fraction of the variance explained by pure noise is, the longer the optimal training period. Overtraining becomes more rare, since the gradient descent tries to update the weights in directions governed mostly by signal and less by noise.

\begin{figure}[h]

\includegraphics[width=0.5\linewidth]{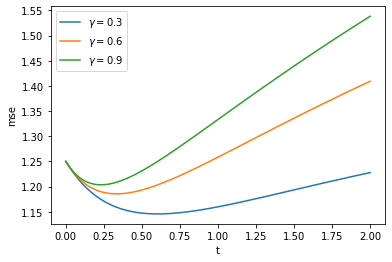}
\caption{Out-of-sample MSE evolution for $\alpha=0.5, H=\delta_1,\lambda=0$ as a function of training time.}

\end{figure}

To examine the tradeoff between $\lambda$ and the optimal stopping time, we plot for $\gamma=\alpha=0.5$ the prediction risk as a function of $t,\lambda$. In this case the optimal $\lambda$ is ${\gamma}{\alpha^{-2}}=2.$ In the plot below we see that close to $\lambda=\lambda^*=2$ training for longer times is required and leads to lower prediction risk values, but when the regularization parameter is much smaller and not close to the optimum it might be very easy to overtrain.

\begin{figure}[h]

\includegraphics[width=0.5\linewidth]{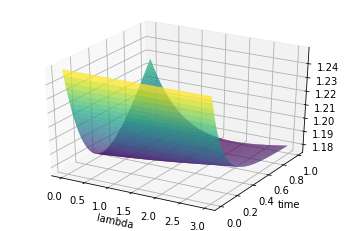}
\caption{MSE surface}
\label{null_surf}
\end{figure}

We also have the following easy consequence, which tells us that for the case of over-regularized ridge regression, under the assumptions we have made, fully training the model is optimal. In other words, the most significant gains from stopping early are for $\lambda$ much smaller than the optimal parameter. %Later in the simulations we are going to carry out this will become more clear.

\begin{corol}\label{over_reg}
For $\lambda \geq {\gamma}{\alpha^{-2}}$ the out-of-sample risk of ridge regression is decreasing as a function of $t.$
\end{corol}

\begin{proof}%[Proof of Corollary \ref{over_reg}]

It is enough to show that for any $x\geq 0,\lambda\geq {\gamma}{\alpha^{-2}}$ $$H(t)=(\sqrt{x}h(x;t,\lambda)-1)^2+\frac{\gamma}{\alpha^2}h(x;t,\lambda)^2$$ is decreasing in $t,$ with $h(x;t,\lambda)$ as in Proposition \ref{ridge_asympt}.  

We have $$H'(t)=2\sqrt{x}\exp(-t(x+\lambda))[x h(x;t,\lambda)-\sqrt{x}+\frac{\gamma}{\alpha^2}h(x;t,\lambda)],$$ so it is enough to show that $h(x;t,\lambda)\leq {\sqrt{x}}(x+\frac{\gamma}{\alpha^2})^{-1}.$ This obviously holds, since $$h(x;t,\lambda)\leq \frac{\sqrt{x}}{x+\lambda}\leq \frac{\sqrt{x}}{x+\frac{\gamma}{\alpha^2}}.$$

\end{proof}

In Figure \ref{fig:ful_two_mass} and \ref{fig:full_auto_regressive} we show how early stopping can tremendously increase performance if the regularization is too small. In fact, in \cite{ali2018continuous} the authors use a coupling argument to provide bounds on optimally tuned versus early-stopped regression and show that the ratio of the risks is at most 1.22 (out-of-sample). In \cite{yao2007early} the authors also derive probabilistic upper bounds for non-parametric regression. Data-dependent stopping rules are studied in \cite{raskutti2014early}. Probabilistic upper bounds have also been studied in . It is also obvious that for $\lambda\geq \lambda^*={\gamma}{\alpha^{-2}}$ early stopping cannot influence the performance.

\bigskip
\begin{figure}[h]

\includegraphics[width=0.5\linewidth]{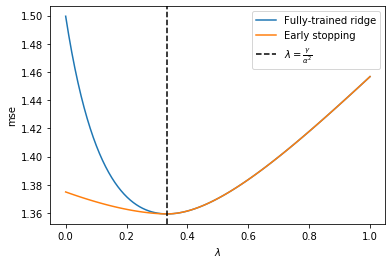}
\caption{Fully-trained ridge regression vs Early Stopping MSE as a function of the regularization parameter $\lambda.$ As in Figure \ref{fig:verif_two_mas} we chose $\gamma={1}/{3},\alpha=1$ and $H=0.5(\delta_1+\delta_4)$}
\label{fig:ful_two_mass}
\end{figure}

\begin{figure}[h]

\includegraphics[width=0.5\linewidth]{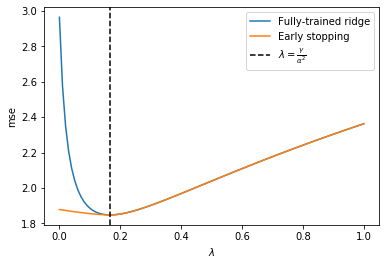}
\caption{Fully-trained ridge regression vs Early Stopping MSE as a function of the regularization parameter $\lambda.$ As in Figure \ref{fig:verif_auto_regress} we chose $\gamma={2}/{3},\alpha=2$ and $\Sigma_{ij}=2^{-\abs{i-j}}$}
\label{fig:full_auto_regressive}
\end{figure}

In Section \ref{app} we prove the analogous result for the training error evolution, which we state next. Here $\underline{F}_{\gamma,H}$ is the companion measure with Stieltjes transform $\underline{m}_{\gamma,H}.$ The main ingredient is again, as in Theorem \ref{main_1}, to control the deviations of the quadratic forms that appear around the mean using the Marcinkiewicz-Zygmund inequality. The means turn out to be simply linear spectral statistics of $S_n$ and the companion matrix $n^{-1}{X^\intercal X},$ hence we can use the Marcenko-Pastur Theorem to derive almost sure limits as $p,n\rightarrow{\infty}.$

\begin{proposition}\label{train_err}
At the point $t$ in time, the training error $E_n(\hat{w})=\displaystyle n^{-1}{\norm{y-X^\intercal \hat{w}}^{2}}$ converges almost surely to 
\begin{equation}\begin{split}
\int \left[\frac{x}{x+\lambda}(1-\exp{(-t(x+\lambda))})-1\right]^2d\underline{F}_{\gamma,H}(x)\\+\alpha^2\int x\left[\frac{x}{x+\lambda}(1-\exp{(-t(x+\lambda))})-1\right]^2 dF_{\gamma,H}(x).
\end{split}
\end{equation}

In particular, at the end of training ($t\rightarrow{\infty}$) it is $$\displaystyle\int \frac{\lambda^2}{(x+\lambda)^2}d\underline{F}_{\gamma,H}+\alpha^2\lambda^2\int \frac{x}{(x+\lambda)^2}dF_{\gamma,H}(x). $$

\end{proposition}
\bigskip

\textbf{Remarks}

The function is decreasing in $t$ as we should expect and increasing in $\alpha.$ For $\lambda\downarrow 0$ we can recover the training error for the unregularized linear regression as $\underline{F}_{\gamma,H}(\{0\})=\max(1-\gamma,0)$ at the end of training regardless of $\Sigma$. In this case $\left(S_n+\lambda\right)^{-1}$ converges to the Moore-Penrose pseudoinverse of the empirical covariance matrix. This explains the severe overfitting phenomenon observed when the number of features exceeds the sample size. Moreover, notice that in the beginning of training, that is when $t\rightarrow{0^+},$ the prediction risk is $1+\alpha^2\int xdF_{\gamma,H}(x)$ for both the training and test sets, which is the variance of the target variable $y$. For the prediction risk, for instance, the surface in Figure \ref{null_surf} starts always close to $1.25$ when $t=0$. 

\bigskip
If we use a different optimization scheme similar estimators arise. For instance, for Nesterov's accelerated gradient method we can use the second order differential equation derived in \cite{su2014differential} to see that the evolution of weights during training is essentially the same as in gradient descent, but with an additional term that is a Bessel function of the empirical covariance matrix $S_n$. Obviously the same method can be applied to provide asymptotics.

%In Section \ref{simul} we carry out extensive simulations to verify our results experimentally.

\section{High-Dimensional Discriminant Analysis}\label{DiscriminantAnalysis}

At this point we study regularization methods for linear discriminant analysis based on random matrix theory. Linear Discriminant Analysis, introduced by Fischer, has been one of the most commonly used classification methods. For example, it has been used in areas such as finance (\cite{LDA_finance}) and biology (\cite{golub1999molecular}). Our methods can potentially  be extended to study regularization in quadratic discriminant analysis, which is again a very popular method used by scientists (for instance, \cite{zhang1997identification}).

First of all, it is essential to derive asymptotic formulas for the classification error for regularized Linear Discriminant Analysis. Once we have done that we will be able to write down an optimization program to find the optimal shrinkage function for the covariance matrix.
\bigskip

\textbf{Classification Model and Assumptions} We will build on the model from \cite{wager}. We assume that we have data generated by two different distributions with equal probabilities. The distributions are assumed to be Gaussian with means $\mu_1=\delta,\mu_2=-\delta\in\mathbb{R}^p.$

\textit{Assumptions on the means}

We assume that the coordinates of $\delta$ are independent and identically distributed random variables with mean 0 and variance ${\alpha^2}{p^{-1}}$. Notice that this is similar to the assumption made in linear regression that each predictor variable has a small effect on the outcome. In particular, each predictor variable contributes a small effect in the distance between class means. In the reference above the authors allow the means to be centered around a point different from zero, but this has no affect on the analysis. We assume that the coordinates of $\sqrt{p}\delta$ have uniformly bounded moments of order $4+\eta$ for some fixed $\eta>0$.
\bigskip

\textit{Assumptions on the covariance matrix}

The two Gaussian distributions have covariance matrix $\Sigma$ which we assume has a spectral distribution that converges weakly to a measure $H$ supported on the real line and there exist deterministic constants $b,B>0$ with $0<b<\lambda_{\min}({\Sigma})\leq \lambda_{\max}(\Sigma)<{B}<\infty$. 

\bigskip
\textit{Labeled data points} 

Suppose that there are $n$ data points $(x_1,y_1),\cdots,(x_n,y_n)$ such that the first ${n}/{2}$ of them, which were generated by the first distribution, have labels $y_i=1$, while the other half have labels $y_i=-1$ and were generated by the second distribution. As always, the setting we are interested in is when $p,n\rightarrow{\infty}$ and ${p}/{n}\rightarrow{\gamma}>0$. We will use the notation $x_i=\Sigma^{{1}/{2}}z_i+\delta y_i$ and $S_n=n^{-1}\sum_{i=1}^n\Sigma^{{1}/{2}}{z_iz_i^\intercal}\Sigma^{{1}/{2}}.$

To classify new data points the Bayes oracle estimator is $\hat{y}(x)=sign(\delta^\intercal\Sigma^{-1}x)$.  Of course, $\Sigma,\delta$ have to be estimated from the data. For $\delta$ we use the estimate $$\displaystyle\hat{\delta}=\frac{1}{n}\sum_{i=1}^{\frac{n}{2}}x_i-\frac{1}{n}\sum_{i=\frac{n}{2}+1}^nx_i.$$ If we replace $\Sigma$ by the empirical covariance matrix $\hat{\Sigma}$, then, as the Marcenko-Pastur theorem indicates, there will be significant noise in the estimation. In addition, if $\gamma>1$, the empirical covariance matrix will not be invertible. For this reason the author in \cite{friedman1989regularized} recommends to use $(\hat{\Sigma}+\lambda)^{-1}$ instead. This is a linear shrinkage method. However, it has been observed (for example, in \cite{lw_2012} and \cite{nercome}) that employing nonlinear shrinkage methods for covariance estimation can greatly enhance the accuracy. For our problem this would mean using the estimator $\hat{y}=sign(\hat{\delta}^\intercal h(\hat{\Sigma})x)$. What is more, it was explained in great detail in \cite{johnstone} that optimal shrinkage depends heavily on the choice of loss function that is used in the estimation. As a consequence, choosing an optimal shrinkage for, say, the Frobenius or Stein loss, does not imply that the shrinkage is optimal for classification purposes. This leads naturally to the questions that we examine in this section, namely we try to answer the following: \textit{How does the choice of shrinkage function affect the classification accuracy? How do well-known shrinkage methods perform asymptotically? What is the optimal shrinkage function for Discriminant Analysis?}
\bigskip

\textbf{Improved Mean Estimation}

First of all, as mentioned in the Introduction, we explain how estimation of ${\delta}$ can be improved by a shrinkage estimator that relies on the sample mean and the sample covariance. To motivate the estimator we impose a Gaussian prior on $\delta.$ Since $y_ix_i|\delta\sim \mathcal{N}(\delta,\Sigma),$ the posterior density of $\delta$ is proportional to $$\exp{\left(-\frac{\sum_{i=1}^n(y_ix_i-\delta)\Sigma^{-1}(y_ix_i-\delta)}{2}-\frac{p\Norm{\delta}^2}{2\alpha^2}\right)},$$ hence $\delta|y_1x_1,\cdots,y_nx_n$ is Gaussian and the posterior mean is given by $$\mathbb{E}[\delta|y_1x_1,\cdots,y_nx_n]=\frac{\alpha^2 n}{p}\left(\Sigma+\frac{n\alpha^2}{p}\right)^{-1}\hat{\delta}.$$ If $\Sigma$ was known we could use this formula, but for unknown $\Sigma$ we might again try to use an estimator of the form $$\hat{\delta}^{(r)}=r(\hat{\Sigma})\hat{\delta}.$$ Surprisingly, if we try to minimize over $r$ the distance $\Norm{\hat{\delta}^{(r)}-\delta}$ we get the following result proved in Section \ref{app}, which is true without the Gaussian assumption on $\delta.$

\begin{proposition}\label{prop:mean_delta}
The continuous function $h$ that asymptotically minimizes $\Norm{\hat{\delta}^{(r)}-\delta}$ is given by $$r(x)=\frac{\frac{\alpha^2}{\gamma}}{\frac{\alpha^2}{\gamma}+\frac{1}{x(f(x)^2+g(x)^2)}},x\in supp(F_{\gamma,H})-\{0\}.$$ If $\gamma>1$ the value of the optimal $r$ at 0 is given by: $$r(0)=\frac{\frac{\alpha^2}{\gamma}}{\frac{\alpha^2}{\gamma}+\frac{1}{(\gamma-1)\underline{m}(0)}}$$
\end{proposition}

The surprising nature of this result is that $r$ comes from plugging in the optimal shrinkage function for covariance estimation in Frobenius loss from \cite{lw_2012} for $\Sigma$ in the posterior mean derived above, although one might expect that trying to approximate $(\Sigma+\gamma^{-1}\alpha^2)^{-1}$ directly might be better. Notice that for the classifier $\hat{y}=sign(\hat{\delta}^\intercal h(\hat{\Sigma})x)$ if $h$ is chosen optimally using $\hat{\delta}$ instead of a better estimate $\hat{\delta}^{(r)}$ there should be no effect in the accuracy, hence the use of $\hat{\delta}$ is justified for the vector of means.

\bigskip

We now return to the problem of the asymptotic classification error. We prove the following result in Section \ref{app} using random matrix theory. This result is of independent mathematical interest, but also the main ingredient in deriving the asymptotic formula of the classification error. Notice that using a polarization argument the result can easily be extended to provide asymptotics for functionals of the form $p^{-1}tr(\Sigma h_1(S_n)\Sigma h_2(S_n)).$

\begin{proposition}\label{Prop_LDA}
There exists a continuous function $K:supp(F_{\gamma,H})\times supp(F_{\gamma,H})\rightarrow{\mathbb{R}}$ such that for any bounded continuous function $h$ defined in an open set containing $supp(F_{\gamma,H})$ we have $$\frac{1}{p}tr\left(\Sigma h\left(S_n\right)\Sigma h\left(S_n\right)\right)\xrightarrow{a.s.}T_{\gamma,H}(h),$$ where \begin{equation}\begin{split} T_{\gamma,H}(h)=  \int \frac{h^2(x)g(x)}{\gamma\pi x^2(f(x)^2+g(x)^2)^2}dx+\int \int K(x,y) h(x)h(y)dx dy,\\
\end{split}\end{equation} if $\gamma<1$. If $\gamma>1$ there is an additional term equal to

\begin{equation}\begin{split}\frac{1}{\gamma}\left[\frac{\underline{m}'(0)}{\underline{m}(0)^4}-\frac{1}{\underline{m}(0)^2}\right]h^2(0)+\frac{2h(0)}{\gamma\underline{m}(0)^2}\int u_h(x)dx\\u_h(x)= \frac{f(x)^2+g(x)^2-2f(x)\underline{m}(0)-x\underline{m}(0)(f(x)^2+g(x)^2)}{\pi x^2(f(x)^2+g(x)^2)^2} g(x)h(x).
\end{split}\end{equation}

The function $K$ is explicitly defined as $$K(x,y)=-\frac{g(x)g(y)}{\gamma\pi^2xy(f(x)^2+g(x)^2)(f(y)^2+g(y)^2)}$$
$$+2\frac{f(x)\left [f(y)^2+g(y)^2 \right]-f(y)\left[f(x)^2+g(x)^2)\right]}{\gamma\pi^2 xy(y-x)\left[f(x)^2+g(x)^2\right]^2\left[f(y)^2+g(y)^2\right]^2}g(x)g(y).$$ %The result is also true if $h$ is bounded and has finitely many discontinuities.

\end{proposition}

\bigskip
\bigskip

When the covariance matrix $\Sigma$ is identity, we can actually see that all terms in $T_{\gamma,I}$ cancel out, except for $ \int {h^2(x)g(x)}/(\gamma\pi x^2(f(x)^2+g(x)^2)^2)dx,$ which simplifies to $ \int h(x) {g(x)}/{\gamma\pi}dx$, and (for $\gamma>1$) ${\gamma^{-1}}\left[(\underline{m}'(0)-\underline{m}(0)^2)/{\underline{m}(0)^4}\right]h^2(0),$ which simplifies to $(1-{\gamma^{-1}})h(0)^2.$ This agrees the limit of $ p^{-1}tr(h(S_n)^2)$ under the standard Marcenko-Pastur distribution. This result is going to play the role of Proposition \ref{gen_lp} in the case of regression, but the proofs for Discriminant Analysis are far more involved. 

We now present our main theorem for this chapter, the proof of which is again in Section \ref{app}. The general strategy in the proof is as follows: Firstly we write down a formula for the classification error that depends only on quadratic forms of $\delta$ and $\hat{\delta}.$ Then we control the deviations of the quadratic forms using again the Marcinkiewicz-Zygmund inequality. Finally, we derive limits for the expectations of the quadratic forms using the generalized Marcenko-Pastur distribution, Proposition \ref{gen_lp} and Proposition \ref{Prop_LDA}.

\begin{theorem}\label{main_discr}
Assume that in order to perform discriminant analysis we use a shrinkage function $h\geq 0$ for the eigenvalues of the sample covariance matrix $\hat{\Sigma}$, which is bounded and continuous in an open interval containing $supp(F_{\gamma,H})$. In particular, we use the classification rule $\hat{y}(x)=sign(\hat{\delta}^\intercal h(\hat{\Sigma})x)$. Then, the out-of-sample classification error converges almost surely to $\Phi(-\sqrt{\Theta(h;\gamma,H))},$ where $\Phi$ is the standard normal c.d.f. and $$\Theta(h)=\frac{\alpha^4\left (\int h(x)dF_{\gamma,H}\right)^2}{\alpha^2 M_{\gamma,H}(h^2)+\gamma T_{\gamma,H}(h)}.$$

\end{theorem}

In real applications we do not know what the actual functions $f,g$ are. However, they can be estimated from the data using kernel estimation. This is explained in \cite{jing2010nonparametric}, where the the authors recommend using a bounded and nonnegative density function with absolutely integrable derivative and window size $h\rightarrow{0}$ that satisfies $\lim nh^{\frac{5}{2}}=\infty.$ As mentioned in the Introduction, in \cite{lw_direct} use a similar method to approximate the optimal shrinkage function for covariance and precision matrix estimation.

We verify the correctness of our theorem in a simulation in Figure \ref{fig:test}. There we plot the classification error as a function of the signal strength $\alpha$ for unregularized and optimally regularized LDA. We see that the empirical estimate of the classification error and the asymptotical value predicted by our theorem are remarkably close. As one might expect, for $\alpha\rightarrow{0}$ all methods give approximately $50\%$ accuracy, while for $\alpha$ large the accuracy of the methods tends to $100\%.$

\begin{figure}[H]
\centering
\begin{subfigure}{.5\textwidth}
  \centering
\includegraphics[trim={0 0 0 1cm},clip,width=1\linewidth]{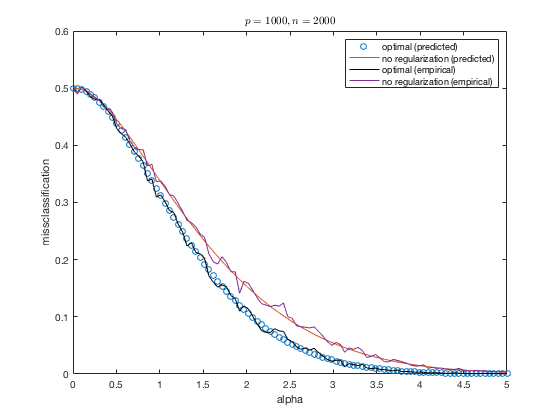}

  \caption{$H=\frac{\delta_1+\delta_{4}}{2}$}
\end{subfigure}%
\begin{subfigure}{.5\textwidth}
  \centering
\includegraphics[trim={0 0 0 1cm},clip,width=1\linewidth]{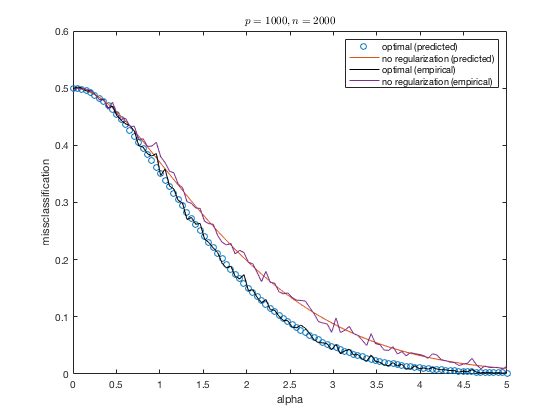}
  \caption{$H=\frac{\delta_1+\delta_2+\delta_3+\delta_4+\delta_5}{5}$}
\end{subfigure}
\caption{Empirical vs theoretical classification error for $p=1000,n=2000$ as a function of the signal strength $\alpha$ for two different choices of limiting spectral distribution $H$ for the population covariance matrix.}
\label{fig:test}
\end{figure}

\subsection{Optimization with respect to the shrinkage function}

Setting $s={\alpha^2}/{\gamma}$ (which plays the role of a signal-to-noise ratio here), we would like to maximize $$\displaystyle \frac{\left(\int h(x)d F_{\gamma,H}(x)\right)^2}{s M_{\gamma,H}(h^2)+T_{\gamma,H}(h)}.$$ Due to the  invariance of this ratio under rescalings of $h$, we can assume that we have fixed $\int hdF_{\gamma,H}=1$ and the problem now becomes

\begin{equation*}
\begin{aligned}
& \underset{h\geq 0}{\text{minimize}}
& & G(h)=sM_{\gamma,H}(h^2)+T_{\gamma,H}(h) \\
& \text{subject to}
& & \int hdF_{\gamma,H}=1.
\end{aligned}
\end{equation*}

The optimization program above describes the tradeoff between the minimizer of $M$ and the minimizer of $T$. In particular, if $\alpha\rightarrow{\infty}$ or $\gamma\rightarrow{0}$, the solution approximates the solution to \begin{equation*}
\begin{aligned}
& \underset{h\geq 0}{\text{minimize}}
& & G(h)=M_{\gamma,H}(h^2) \\
& \text{subject to}
& & \int hdF_{\gamma,H}=1.
\end{aligned}
\end{equation*}

Up to rescalings, $h(x)$ for $x>0$ is the same as the optimal shrinkage for covariance estimation, as wee see by examining when equality holds in the Cauchy-Schwartz: $$\int \frac{h^2g}{\gamma \pi x(f^2+g^2)}dx\int x(f^2+g^2)\frac{g}{\gamma\pi}dx\geq \left(\int \frac{hg}{\gamma\pi}dx\right)^2.$$ This proves the claim for $\gamma<1.$ For $\gamma>1$ we also need to take into account the value of $h(0),$ but it is straightforward to extend the argument.

%When $\alpha\rightarrow{0}$, which implies $s\rightarrow{0}$, the goal is to minimize $T_{\gamma,H}(h)$ under the restrictions imposed. 

We conclude that, when the signal is strong, the optimal shrinkage function for covariance estimation in Frobenius norm is approximately the same as the optimal shrinkage function for linear discriminant analysis. When the signal is low, other shrinkage functions might perform better. All things considered, the last theorem vastly generalizes the results of the references mentioned and allows us to compute sharp asymptotics for the accuracy of complicated shrinkage estimators like the ones from \cite{ledoit2011eigenvectors}, when applied to classification via discriminant analysis. %In Section \ref{simul} we are going to study these findings empirically. %the optimal shrinkage for regularization in discriminant analysis with the optimal shrinkage for covariance and precision estimation mentioned in the reference above.

 We now for simplicity restrict our attention to $\gamma<1$ and study more thoroughly the behaviour of the optimal shrinkage function. Since $$\frac{1}{p}tr\left(\Sigma h(S_n)\Sigma h(S_n)\right)\geq \left(\frac{tr\left(\Sigma h(S_n)\right)}{p}\right)^2\xrightarrow{a.s.} (M_{\gamma,H}(h))^2,$$ we can write down the following relaxation:
 
 \begin{equation*}
\begin{aligned}
& \underset{h\geq 0}{\text{minimize}}
& & M_{\gamma,H}(h^2) +\frac{\gamma}{\alpha^2}(M_{\gamma,H}(h))^2\\
& \text{subject to}
& & \int hdF_{\gamma,H}=1.
\end{aligned}
\end{equation*}

This optimization problem can easily be solved analytically. Below we make an interesting remark about the solution: If $h_0$ is the solution to this optimization problem, we have for any $u$ such that $\int u dF_{\gamma,H}=0$ 
\begin{equation}\label{eq:relax}\begin{split}\frac{d}{dt}M_{\gamma,H}((h_0+tu)^2) +\frac{\gamma}{\alpha^2}(M_{\gamma,H}(h_0+tu))^2|_{t=0}=0\\
\Rightarrow{\int \frac{h_0ug}{\gamma\pi x (f^2+g^2)}}dx+\frac{\gamma}{\alpha^2}\int \frac{h_0g}{\gamma \pi x(f^2+g^2)}dx \int \frac{ug}{\gamma\pi x (f^2+g^2)}dx=0\end{split}
\end{equation}

We conclude that, since this holds for all $u$ with $\int u{g} dx=0,$ it must be that for some constant $A$ $$\frac{g}{\gamma\pi}\left[\frac{h_0}{x(f^2+g^2)}+\frac{\gamma}{\alpha ^2} \frac{1}{x(f^2+g^2)}\int \frac{h_0g}{\gamma\pi x (f^2+g^2)}dx\right]=A\frac{g}{\gamma \pi}.$$ In other words, for suitably chosen constants $A,B$ we have for all $x\in supp(F_{\gamma,H})$ $$h_0(x)=Ax(f(x)^2+g(x)^2)-B.$$ This is the optimal shrinkage for covariance estimation with shift and rescaling. Since rescalings of $h$ do not affect the classification error we can take $A=1$. After this the solution can be viewed as a combination of ridge regularization (with negative parameter) and the optimal shrinkage for covariance estimation in Frobenius norm. %This also gives a quick way to approximate the optimal shrinkage only at the cost of tuning the ridge parameter in ridge-regularized Linear Discriminant Analysis. To be more specific, since rescalings of $h$ are not important we can only tune $B$ assuming $A=1,$ for example via cross-validation.

\textbf{Solving the optimization problem} To solve for $h$ numerically, we may approximate $h$ by piecewise constant functions defined on small intervals that cover $supp(F_{\gamma,H})$. The problem then becomes a convex quadratic program in the values assigned to each interval. In fact, the condition $h\geq 0$ should be redundant, as the optimal shrinkage function should be positive. This procedure,  of course, does not impose continuity, but, if a continuous minimizer exists, one might expect that this procedure should approximate that.  In addition, it is possible to modify Theorem \ref{main_discr} to account for problems with finitely many discontinuities, but at this point we ignore this detail.

\bigskip
\textbf{Estimation of $\alpha^2$}

In order to solve numerically for the optimal shrinkage function in practice one needs to estimate $f,g$ and $\alpha^2.$ As mentioned $f,g$ can be estimated using kernel estimation. We now explain the missing part, namely estimation of $\alpha^2.$
With the assumptions made it is immediate to see that $$\Norm{\hat{\delta}}^2-\alpha^2-\frac{tr(\Sigma)}{n}\xrightarrow{a.s.}0.$$ Since $n^{-1}(tr(\Sigma)-tr(\hat{\Sigma}))\xrightarrow{a.s.}0,$ we conclude that $$\Norm{\hat{\delta}}^2-\frac{tr(\hat{\Sigma})}{n}\xrightarrow{a.s.}\alpha^2.$$

We compare the asymptotic missclassification risk for limiting spectral distribution $H=0.5(\delta_{0.75}+\delta_{15})$ and $\gamma=0.75.$ The estimators that we consider are the optimal shrinkage estimator that minimizes the asymptotic risk, the optimal shrinkage for covariance and precision estimation of \cite{ledoit2011eigenvectors} with respect to the Frobenius loss (the latter being $h(\lambda)=(\gamma-1-2\lambda f(\lambda))/{\lambda}$) -which we call ledoit-peche 1 and ledoit-peche 2 respectively- the optimal ridge-regularized estimator of the form $(\hat{\Sigma}+\lambda)^{-1}$ and the unregularized sample covariance estimator.

\begin{figure}[H]
    \centering
    \includegraphics[trim={0 0 0 1.5cm},clip,width=0.55\linewidth]{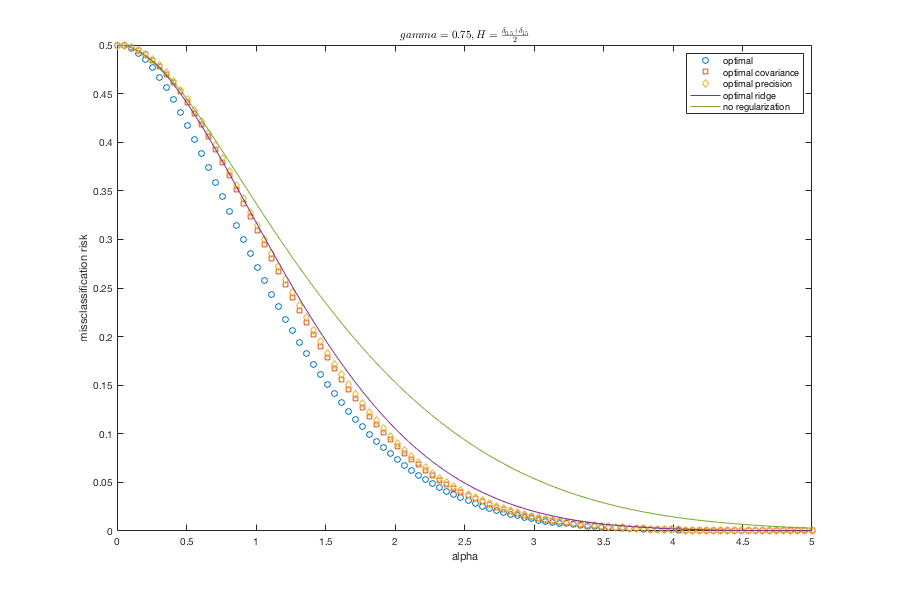}
    \caption{Comparison of different shrinkage estimators for LDA for $\gamma=0.75,H=0.5(\delta_{0.5}+\delta_{15}).$}
    \label{fig:lda_comparisons}
\end{figure}

We observe that the optimal covariance and optimal precision shrinkage functions perform very similarly and slightly better than optimally ridge-regularized shrinkage. The worst is unregularized, as one might expect.

\section{Conclusions and Future Work}\label{future}

In the first part of the paper we studied general singular value shrinkage methods for regularization of regression in high-dimensional statistics. For this purpose we derived results of independent mathematical interest for the convergence of some trace functionals that include both the population and the sample covariance matrices, as in Theorem \ref{main_1}. Before the development of these tools this was only possible for the case of ridge regression. We also presented important applications of Theorem \ref{main_1}. Firstly, we explained how it can be used to recover the optimal shrinkage function for covariance estimation derived heuristically in \cite{ledoit2011eigenvectors} and used in \cite{lw_2012}. Secondly, we proved the optimality of ridge regression with parameter ${\gamma}/{\alpha^2}$ in this setting among all shrinkage-based methods for general coefficient distributions. Furthermore, our results allowed us to analyze methods such as early stopping in models with arbitrary predictor covariance matrices. In particular, we studied the test error evolution in gradient descent training, which gave us important insights about the early stopping regularization methods used in the deep learning community and by practitioners. Even in this very simplistic setting the regime we are studying explains some attributes of the overtraining phenomenon seen in much more complex models. Within our model  early stopping may provide important benefits and achieve performance comparable to optimally tuned ridge regression for parameters much smaller than the optimal. As a direction of future work, it might be interesting to examine if similar results could be established for larger classes of models, for example for shallow neural networks. This will be a step closer to the extremely complex models that have been used in the last decade in the most successful applications of deep learning.

In the second part of the paper we provided sharp asymptotic formulas for regularization methods based on singular value shrinkage in high-dimensional classification. We were able to analyze asymptotically the performance of shrinkage methods proposed in the literature for covariance estimation (such as those in \cite{lw_2012}) in the case of discriminant analysis problems. With the tools previously available from random matrix theory this had been possible only for ridge-regularized linear discriminant analysis, so, to extend them, we had to derive results of independent mathematical interest. Our theorems also led to a procedure for finding numerically the optimal shrinkage function for classification in linear discriminant analysis. Finally, we examined and verified our results in extensive simulations. Extending our results to the case of quadratic discriminant analysis would be a natural next step. As a closely related direction for future work it might be interesting to explore how these results can be extended to other methods that are used frequently, such as logistic regression or support vector machines. As mentioned in the Introduction this question was considered in \cite{taheri2020sharp}, but only for standard Gaussian features. Similar questions could be asked for more general covariance matrices and distributions such as those in our paper. Answering these questions will allow us to understand better how these commonly used methods generalize in very high dimensional settings.

\bigskip

%\appendix

%\setcounter{secnumdepth}{0}
\section{Proofs of Main Results}\label{app}

\subsection{Proofs for Regression}

We will need the following lemma which is adapted from Lemma 7.8, Lemma 7.9 and Lemma 7.10 from \cite{erdHos2017dynamical}.

\begin{lemma}\label{lp_lemma}
%Let $Z\in\mathbb{R}^N$ be a vector with i.i.d. entries with mean 0, variance 1 and $12$-th moment bounded by some constant $B$ independent of $N$. Let also $A$ be a given $N\times N$ matrix. Then, there exists a  constant $K>0$ (that depends only on $B$) with $$\mathbb{E}\left[\left(Z^\intercal AZ-tr(A)\right)^6\right]\leq K\norm{A}^6N^3.$$
Let $q\geq 2$ and $X_1,\cdots,X_N,Y_1,\cdots,Y_N$ be independent random variables with mean 0, variance 1 and $2q$-th moment bounded by $c_0$. Then, for any deterministic $(b_i)_{1\leq i\leq N}, (a_{ij})_{1\leq i,j\leq N}$ we have for some positive constant $C_q=C_q(c_0)$:

\begin{equation}
\Norm{\sum_i b_i (X_i^2-1)}_q\leq C_q (\sum_i \abs{b_i}^2)^{\frac{1}{2}}
\end{equation}

\begin{equation}
\Norm{\sum_{i,j}a_{ij}X_iY_j}_q \leq C_q (\sum_{i,j}a_{ij}^2)^{\frac{1}{2}}
\end{equation}

\begin{equation}
\Norm{\sum_{i\neq j}a_{ij}X_iX_j}_q\leq C_q (\sum_{i\neq j}a_{ij}^2)^{\frac{1}{2}}
\end{equation}
\end{lemma}

%A similar result is Proposition 4.1 in \cite{coul}, which we state below.

\begin{proof}[Proof of Theorem \ref{main_1}]

Let us define $$M_1=\sum_{i=1}^p\sqrt{\lambda_i}h(\lambda_i)u_iu_i^\intercal, M_2=\sum_{i=1}^ph(\lambda_i)u_iv_i^\intercal.$$

We have $$\hat{w}=\sum_{i=1}^ph(\lambda_i)u_iv_i^\intercal \frac{y}{\sqrt{n}}=\left[\sum_{i=1}^p h(\lambda_i)u_iv_i^\intercal\right] \frac{X^\intercal w+\varepsilon}{\sqrt{n}}$$

$$=\sum_{i=1}^p \sqrt{\lambda_i}h(\lambda_i)u_iu_i^\intercal w +\sum_{i=1}^ph(\lambda_i) u_iv_i^\intercal \frac{\varepsilon}{\sqrt{n}}=M_1w+M_2\frac{\varepsilon}{\sqrt{n}}.$$

The prediction risk is equal to $$1+(\hat{w}-w)^\intercal \Sigma (\hat{w}-w)=1+\left[(M_1-I_p)w+M_2\frac{\varepsilon}{\sqrt{n}}\right]^\intercal \Sigma \left[(M_1-I_p)w+M_2\frac{\varepsilon}{\sqrt{n}}\right]$$

$$=1+w^\intercal (M_1-I_p)\Sigma (M_1-I_p)w+\frac{1}{n}\varepsilon^\intercal M_2^\intercal \Sigma M_2\varepsilon+2\frac{\varepsilon^\intercal}{\sqrt{n}}M_2^\intercal\Sigma (M_1-I_p)w.$$

According to Lemma \ref{lp_lemma} for $q={(4+\eta)}/{2}$ we have \begin{equation}\begin{split}
\mathbb{E}\left[\left(w^\intercal (M_1-I_p)\Sigma (M_1-I_p)w -\frac{\alpha^2}{p}tr\left((M_1-I_p)^2\Sigma \right)\right)^q|M_1\right]=\\\mathcal{O} \left(\Norm{(M_1-I_p)\Sigma (M_1-I_p)}_F^q\frac{\alpha^{2q}}{p^q}\right)
=\mathcal{O}\left(p^{-\frac{q}{2}}\Norm{(M_1-I_p)\Sigma(M_1-I_p)}^q\right).
\end{split}
\end{equation} Since $\displaystyle \limsup \max_{1\leq i\leq p}\lambda_i=\limsup \Norm{S_n}\leq\limsup \norm{\Sigma}\norm{n^{-1}{ZZ^\intercal}}\leq \limsup{\norm{\Sigma}}(1+\sqrt{\gamma})^2$ almost surely, we conclude that almost surely $\sup_p \norm{\Sigma}\norm{M_1-I_p}<\infty.$ By Borel-Cantelli, since ${q}>2$, this implies that $$w^\intercal (M_1-I_p)\Sigma (M_1-I_p)w-\frac{\alpha^2}{p}tr\left((M_1-I_p)^2\Sigma\right)\xrightarrow{a.s.}0.$$

A similar argument shows that $$\frac{1}{n}\varepsilon^\intercal M_2^\intercal \Sigma M_2\varepsilon-\frac{1}{n}tr\left(M_2^\intercal \Sigma M_2\right)\xrightarrow{a.s.}0.$$

By Lemma \ref{lp_lemma} for $q=(4+\eta)/{2}$ and  $n^{-1/2}{\varepsilon^\intercal}M_2^\intercal\Sigma (M_1-I_p)w$ we get as before: $$\mathbb{E}\left[\left(\frac{\varepsilon^\intercal}{\sqrt{n}}M_2^\intercal\Sigma(M_1-I_p)w\right)^{q}|M_1,M_2\right]=\mathcal{O}\left(\frac{\alpha^q}{(pn)^{\frac{q}{2}}}\left[\norm{M_2^\intercal \Sigma (M_1-I_p)}_{F}^q|M_1,M_2\right]\right)$$
$$=\mathcal{O}\left(\frac{\alpha^q}{n^{\frac{q}{2}}}\norm{M_2^\intercal\Sigma (M_1-I_p)}^q\right).$$ So, again by Borell-Cantelli, we derive that $$\frac{\varepsilon^\intercal}{\sqrt{n}}M_2^\intercal \Sigma (M_1-I_p)w\xrightarrow{a.s.}0.$$

Finally, we see by Proposition \ref{gen_lp} $$\frac{1}{p}tr\left((M_1-I_p)^2\Sigma\right)=\frac{1}{p}tr\left(\left(h\left(S_n\right)\left(S_n\right)^{\frac{1}{2}}-I_p\right)^2\Sigma\right)\xrightarrow{a.s.}M_{\gamma,H}((\sqrt{x}h(x)-1)^2)$$
$$\frac{1}{n}tr\left( M_2^\intercal \Sigma M_2\right)=\frac{1}{n}tr\left(\Sigma M_2M_2^\intercal\right)=\frac{p}{n}\frac{1}{p}tr\left(\Sigma h\left(S_n\right)^2\right)\xrightarrow{a.s.}\gamma M_{\gamma,H}(h(x)^2).$$

Combining everything we get that the risk converges almost surely to $$M_{\gamma,H}((\sqrt{x}h(x)-1)^2)+\gamma M_{\gamma,H}(h(x)^2).$$

The proof is completed.

\end{proof}

The next result that we prove is the asymptotic formulas for the training error, namely Proposition \ref{train_err}.

\begin{proof}[Proof of Proposition \ref{train_err}]We have $$\frac{\norm{y-X^\intercal \hat{w}}^2}{n}=\frac{\norm{\varepsilon+X^\intercal w-X^\intercal \hat{w}}^2}{n}.$$

We have 
$
-\varepsilon+X^\intercal(\hat{w}(t)-w)=M_x w+M_{\varepsilon}\varepsilon,$

%=-\varepsilon+X^\intercal\left[\left(I-\exp\left(-t\left(\frac{XX^\intercal}{n}+\lambda\right)\right)\right)\left(\frac{XX^\intercal}{n}+\lambda\right)^{-1}\left(\frac{XX^\intercal}{n}w+\frac{X\varepsilon}{n}\right)-w\right]

 where

\begin{equation}\begin{split}\label{mx}
M_x=X^\intercal\left(I-\exp\left(-t\left(\frac{XX^\intercal}{n}+\lambda\right)\right)\right)\left(\frac{XX^\intercal}{n}+\lambda\right)^{-1}\frac{XX^\intercal}{n}-X^\intercal\\
=\sqrt{n}\sum_{i=1}^p\left(\sqrt{\lambda_i}\frac{1-\exp{(-t(\lambda_i+\lambda))}}{\lambda_i+\lambda}\lambda_i-\sqrt{\lambda_i}\right)v_i u_i^\intercal.
\end{split}
\end{equation}

\begin{equation}\begin{split}\label{me}
M_{\varepsilon}=-I_n+X^\intercal\left(I-\exp\left(-t\left(\frac{XX^\intercal}{n}+\lambda\right)\right)\right)\left(\frac{XX^\intercal}{n}+\lambda\right)^{-1}\frac{X}{n}\\
=-1+\sum_{i=1}^p\left(\lambda_i\frac{1-\exp{(-t(\lambda+\lambda_i))}}{\lambda_i+\lambda}\right)v_iv_i^\intercal.
\end{split}
\end{equation}

Thus, the training error is $$E_n(\hat{w})=\frac{w^\intercal M_{x}^\intercal M_x w+2w^\intercal M_x^\intercal M_{\varepsilon}\varepsilon+\varepsilon^\intercal M_{\varepsilon}^\intercal M_{\varepsilon}\varepsilon}{n}.$$

The exact same argument based on Lemma \ref{lp_lemma} that we used in the previous proof implies that $$\frac{w^\intercal M_x^\intercal M_x w}{n}-\frac{\alpha^2}{p}tr\left(\frac{M_x^\intercal M_x}{n}\right)\xrightarrow{a.s.}0$$

$$\frac{w^\intercal M_x^\intercal M_{\varepsilon}\varepsilon}{n}\xrightarrow{a.s.}0$$

$$\frac{\varepsilon^\intercal M_{\varepsilon}^\intercal M_{\varepsilon}\varepsilon}{n}-\frac{1}{n}tr\left(\frac{M_{\varepsilon}^\intercal M_{\varepsilon}}{n}\right)\xrightarrow{a.s.}0$$

To finish the proof we use the fact that $$\frac{1}{p}tr\left(\frac{M_x^\intercal M_x}{n}\right)\xrightarrow{a.s.}\int x \left[(1-\exp{(-t(x+\lambda))})\frac{x}{x+\lambda}-1\right]^2 dF_{\gamma,H}(x),$$ which is clear from \eqref{mx} and the Marcenko-Pastur theorem, since $u_i$'s are the eigenvectors of ${XX^\intercal}/{n}$ and $\lambda_i$'s the corresponding eigenvalues.
Furthermore, $$\frac{1}{n}tr\left(\frac{M_{\varepsilon}^\intercal M_{\varepsilon}}{n}\right)\xrightarrow{a.s.}\int \left[(1-\exp{(-t(x+\lambda))})\frac{x}{x+\lambda}-1\right]^2 d\underline{F}_{\gamma,H}(x),$$ which follows from \eqref{me} by the same reasoning.

\end{proof}

\subsection{Proofs for Linear Discriminant Analysis}

Now we focus on the results that concern Discriminant Analysis. The assumptions are the same as in Section \ref{DiscriminantAnalysis}. First of all, we prove Proposition \ref{prop:mean_delta} for estimation of the mean vector.

\begin{proof}[Proof of Proposition \ref{prop:mean_delta}]
We will first prove the result for $\gamma<1,$ since minor changes are required to account for the mass at $0$ when $\gamma>1.$
We have $\hat{\delta}=\Sigma^{{1}/{2}}n^{-{1}/{2}}u+\delta,$ where $u\sim\mathcal{N}(0,I_n)$ is independent of the sample covariance matrix $\hat{\Sigma}.$ Using Lemma \ref{lp_lemma} we have: $$\Norm{\hat{\delta}^{(r)}-\delta}^2=\Norm{r(\hat{\Sigma})\left(\frac{\Sigma^{\frac{1}{2}}u}{\sqrt{n}}+\delta\right)-\delta}^2=\frac{tr\left(\Sigma r^2(\hat{\Sigma})\right)}{n}+\frac{\alpha^2}{p}\Norm{r(\hat{\Sigma})-I_p}_F^2+\omicron{(1)},$$ which (by Proposition \ref{gen_lp} and the Marcenko-Pastur theorem converges almost surely to $$\gamma M(r^2)+\alpha^2\int (r(x)-1)^2dF_{\gamma,H}=\int \gamma \frac{r(x)^2}{x(f(x)^2+g(x)^2)}+\alpha^2(r(x)-1)^2dF_{\gamma,H}.$$ The integrand is quadratic in $r$ and setting the derivative equal to 0 we see that it is minimized exactly for the function $r$ that we claimed. If $\gamma>1$ the only change in the proof is that we also need to minimize the contribution at 0, which equals $$\frac{r^2(0)}{\underline{m}(0)}+\alpha^2\frac{\gamma-1}{\gamma}(r(0)-1)^2.$$
\end{proof}

\bigskip

We now prove Proposition \ref{modif_lda}, a very important result that we will need.

\begin{proposition}\label{modif_lda}
Let $z_1,z_2\in\mathbb{C}^+$. Then, $$\frac{1}{p}tr\left(\Sigma \left(S_n-z_1\right)^{-1}\Sigma \left(S_n-z_2\right)^{-1}\right)\xrightarrow{a.s.}$$
$$-\frac{1}{\gamma z_1z_2\underline{m}(z_1)\underline{m}(z_2)}+\frac{\underline{m}(z_2)-\underline{m}(z_1)}{\gamma z_1z_2\underline{m}(z_1)^2\underline{m}(z_2)^2 (z_2-z_1)}.$$
\end{proposition}

\begin{proof}
Our proof is similar to the proof of the main theorem in \cite{ledoit2011eigenvectors}, but it is way more complicated due to the nature of the functionals involved in the statement. %We set $x_i=\Sigma^{\frac{1}{2}}z_i.$

Let $$R(z)=\left(S_n-z\right)^{-1}, R_{jk}(z)=\left(\frac{1}{n}\sum_{i\neq j,k}x_ix_i^\intercal-z\right)^{-1}.$$

We also define $$a_{kj}=tr\left(\frac{x_kx_k^\intercal}{n}R(z_1)\frac{x_jx_j^\intercal}{n}R(z_2)\right).$$

\textbf{The case $j=k.$}

For $j=k$, we have $$a_{jj}=\frac{x_j^\intercal R(z_1)x_j}{n}\frac{x_j^\intercal R(z_2)x_j}{n}.$$ Using $R_{jj}(z_1)-R(z_1)=R_{jj}(z_1)\frac{x_jx_j^\intercal}{n}R{(z_1)}$ we derive $$\frac{x_j^\intercal R(z_1)x_j}{n}=1-\frac{1}{1+\frac{x_j^\intercal R_{jj}x_j}{n}}.$$ In the proof of Lemma 2.2 in the reference above the authors prove that $$\max_{j\in \{1,\cdots,n\}}\abs{\frac{x_j^\intercal R_{jj}(z_1)x_j}{n}-\frac{1}{n}tr\left(R_{jj}(z_1)\Sigma\right)}\xrightarrow{a.s.}0.$$

As a result, since by the reference above $$\max_{1\leq j\leq n}\abs{\frac{1}{n}tr\left(R_{jj}(z_1)\Sigma\right)+1+\frac{1}{z_1\underline{m}(z_1)}}\xrightarrow{a.s.}0,$$ we have $$\max_{j\in\{1,\cdots,n\}}\abs{a_{jj}-(1+z_1\underline{m}(z_1))(1+z_2\underline{m}(z_2))}\xrightarrow{a.s.}0.$$ We conclude that

\begin{equation}\label{term1}\frac{1}{p}\sum_{j=1}^na_{jj}\xrightarrow{a.s.}\frac{1}{\gamma}(1+z_1\underline{m}(z_1))(1+z_2\underline{m}(z_2)).
\end{equation}

\textbf{The case $j\neq k.$}

For $j\neq k$ we have $$R(z_1)-R_{kj}(z_1)=-n^{-1}R(z_1){x_kx_k^\intercal}R_{kj}(z_1)-n^{-1}R(z_1){x_jx_j^\intercal}R_{kj}(z_1)$$ and, consequently,
\begin{equation}\begin{split}
{\frac{x_k^\intercal R(z_1)x_j}{n}-\frac{x_k^\intercal R_{kj}(z_1)x_j}{n}}\\=-{\frac{x_k^\intercal R(z_1)x_k}{n}\frac{x_k^\intercal R_{kj}(z_1)x_j}{n}-\frac{x_k^\intercal R(z_1) x_j}{n}\frac{x_j^\intercal R_{kj}(z_1)x_j}{n}}\\\Rightarrow{\frac{x_k^\intercal R(z_1)x_j}{n}}=\frac{\frac{x_k^\intercal R_{kj}(z_1)x_j}{n}(1-\frac{x_k^\intercal R(z_1)x_k}{n})}{1+\frac{x_j^\intercal R_{kj}(z_1)x_j}{n}}\\=\frac{x_k^\intercal R_{kj}(z_1)x_j}{n}\frac{1}{\left(1+\frac{x_j^\intercal R_{kj}(z_1)x_j}{n}\right)\left(1+\frac{x_k^\intercal R_{kk}(z_1)x_k}{n}\right)}.
\end{split}\end{equation}

We conclude that, for $j\neq k$, 
\begin{equation}\begin{split}\label{ajnotk}
a_{kj}=\frac{x_k^\intercal R_{kj}(z_1)x_j}{n}\frac{x_j^\intercal R_{kj}(z_2)x_k }{n} A_{kj}=tr\left(\frac{x_kx_k^\intercal}{n}R_{kj}(z_1)\frac{x_jx_j^\intercal}{n}R_{kj}(z_2)\right)A_{kj},\\
A_{kj}=\frac{1}{\left(1+\frac{x_j^\intercal R_{kj}(z_1)x_j}{n}\right)\left(1+\frac{x_k^\intercal R_{kk}(z_1)x_k}{n}\right)}\frac{1}{\left(1+\frac{x_j^\intercal R_{kj}(z_2)x_j}{n}\right)\left(1+\frac{x_k^\intercal R_{kk}(z_2)x_k}{n}\right)}.
\end{split}\end{equation}

%As above we have that $$\max_{j\in\{1,\cdots,n\}}\abs{A_{ij}-z_1^2z_2^2\underline{m}(z_1)^2\underline{m}(z_2)^2}\xrightarrow{a.s.}0,$$ since 

Using the Hanson-Wright inequality from \cite{vershynin_hanson_wright} for the real and imaginary parts of $R_{kj}(z_1)$ we get that there exist absolute constants $C,c>0$ such that $$\mathbb{P}(\abs{\frac{x_j^\intercal R_{k j}(z_1)x_j}{n}-\frac{1}{n}tr\left(\Sigma R_{kj}(z_1)\right)}>t|R_{kj})\leq C\exp{\left(-c\min\left(\frac{n^2t^2}{\norm{R_{j k}}_F^2},\frac{nt}{\norm{R_{jk}}}\right)\right)}.$$

Since $\norm{R_{kj}}\leq {Im(z_1)^{-1}},$ we conclude by the union bound that $$\mathbb{P}(\max_{1\leq k,j\leq n}\abs{\frac{x_j^\intercal R_{kj}(z_1)x_j}{n}-\frac{1}{n}tr\left(\Sigma R_{kj}(z_1)\right)}\geq t)$$
$$\leq Cn^2\exp{(-c\min\left(Im(z_1)^2nt^2,Im(z_1)nt\right))}.$$ We conclude that \begin{equation}\begin{split}\label{const1}
\max_{1\leq k,j\leq n}\abs{\frac{x_j^\intercal R_{kj}(z_1)x_j}{n}-\frac{1}{n}tr\left(\Sigma R_{kj}(z_1)\right)}\xrightarrow{a.s.}0.
\end{split}\end{equation}

By Lemma 2.6 in \cite{silverstein1995empirical} we have $$\frac{1}{n}\abs{tr\left(\Sigma R(z_1)-\Sigma R_{kk}(z_1)\right)}\leq \frac{\norm{\Sigma}}{nIm(z_1)}.$$ Applying this if we omit $x_k$ we get $$\frac{1}{n}\abs{tr\left(\Sigma R_{kk}(z_1)-\Sigma R_{kj}(z_1)\right)}\leq \frac{\norm{\Sigma}}{nIm(z_1)}.$$

Combining the last two inequalities gives
\begin{equation}\label{const2}
\frac{1}{n}\abs{tr(\Sigma R(z_1)-\Sigma R_{kj}(z_1))}\leq \frac{2\norm{\Sigma}}{nIm(z_1)}.
\end{equation}

Using again $$\frac{1}{n}tr(\Sigma R(z_1))\xrightarrow{a.s.}-\frac{1+z_1\underline{m}(z_1)}{z_1\underline{m}(z_1)},$$ \eqref{const1} and \eqref{const2} imply $$\max_{1\leq k,j\leq n}\abs{\frac{x_j^\intercal R_{kj}(z_1)x_j}{n}+1+\frac{1}{z_1\underline{m}(z_1)}}\xrightarrow{a.s.}0.$$

Of course, the same convergence is true if we replace $z_1$ by $z_2.$ We conclude that for $j\neq k$ 

\begin{equation}\label{Aij}
\max_{1\leq k,j\leq n}\abs{A_{kj}-z_1^2z_2^2\underline{m}(z_1)^2\underline{m}(z_2)^2}\xrightarrow{a.s.}0.\end{equation}

We now use the Hanson-Wright inequality twice and the exact same argument as above to get

\begin{equation}\label{const_last}
\max_{1\leq k\leq n}\abs{\frac{1}{p}tr\left(\sum_{j=1,j\neq k}^nR_{kj}(z_1)\frac{x_jx_j^\intercal}{n}R_{kj}(z_2)x_kx_k^\intercal\right)-\frac{1}{pn}tr\left(\Sigma R(z_1)\Sigma R(z_2)\right)}\xrightarrow{a.s.}0.
\end{equation}

Combining \eqref{Aij} and \eqref{const_last} we get \begin{equation}\label{term2}
\frac{1}{p}\sum_{k\neq j}a_{kj}=z_1^2z_2^2\underline{m}(z_1)^2\underline{m}(z_2)^2\frac{1}{p}tr\left(\Sigma R(z_1)\Sigma R(z_2)\right)+\omicron{(1)}.
\end{equation}

Let $$E=\frac{1}{p}\sum_{1\leq k,j\leq n}a_{kj}=\frac{1}{p}tr\left((I_p+z_1R(z_1)(I_p+z_2R(z_2)))\right).$$ Then, by the Marcenko-Pastur theorem we know that $$E\xrightarrow{a.s.}1+z_1m(z_1)+z_2m(z_2)+\frac{z_1z_2(m(z_2)-m(z_1))}{z_2-z_1}.$$ Notice that here we used the partial fraction decomposition ${(x-z_1)^{-1}(x-z_2)^{-1}}=(z_2-z_1)^{-1}\left({(x-z_2)^{-1}}-{(x-z_1)^{-1}}\right),$ which immplies that $${R(z_1)R(z_2)=\frac{1}{z_2-z_1}\left(R(z_2)-R(z_1)\right)}.$$

We conclude by \eqref{term1} and \eqref{term2} that $$E=\frac{1}{\gamma}(1+z_1\underline{m}(z_1))(1+z_2\underline{m}(z_2))+z_1^2z_2^2\underline{m}(z_1)^2\underline{m}(z_2)^2\frac{1}{p}tr\left(\Sigma R(z_1)\Sigma R(z_2)\right)+\omicron{(1)}.$$ The last two equations imply that \begin{equation}\begin{split}\label{master1}
\lim_{p\rightarrow{\infty}}z_1^2z_2^2\underline{m}(z_1)^2\underline{m}(z_2)^2\frac{1}{p}tr\left(\Sigma R(z_1)\Sigma R(z_2)\right)\\
=-\frac{1}{\gamma}(1+z_1\underline{m}(z_1))(1+z_2\underline{m}(z_2))+1+z_1m(z_1)+z_2m(z_2)+\frac{z_1z_2(m(z_2)-m(z_1))}{z_2-z_1}.
\end{split}\end{equation}

Using the equation for the companion Stieltjes transform $$\gamma(1+zm(z))=1+z\underline{m}(z),$$ we can rewrite \eqref{master1} after a straightforward manipulation: \begin{equation}\label{master2}\begin{split}\lim_{p\rightarrow{\infty}}z_1^2z_2^2\underline{m}(z_1)^2\underline{m}(z_2)^2\frac{1}{p}tr\left(\Sigma R(z_1)\Sigma R(z_2)\right)\\
=-\frac{1}{\gamma}z_1z_2\underline{m}(z_1)\underline{m}(z_2)+\frac{z_1z_2(\underline{m}(z_2)-\underline{m}(z_1))}{\gamma (z_2-z_1)}.
\end{split}\end{equation}

This completes the proof.
\end{proof}

This gives us the main tool in proving Proposition \ref{Prop_LDA}.

\begin{proof}[Proof of Proposition \ref{Prop_LDA}]
As in Proposition \ref{gen_lp} we only need to prove the result for analytic functions $h$ in some open set that contains $supp(F_{\gamma,H}).$
We write using the Cauchy integral formula:$$\frac{1}{p}tr\left(\Sigma h\left(S_n\right)\Sigma h\left(S_n\right)\right)=\left(\frac{1}{2\pi i}\right)^2\oint_{\Gamma}\oint_{\Gamma'}h(z_1)h(z_2)F_{p}(z_1,z_2)dz_1dz_2,$$ where

$$F_p(z_1,z_2)=\frac{1}{p}tr\left(\Sigma\left(S_n-z_1\right)^{-1}\Sigma \left(S_n-z_2\right)^{-1}\right)$$ and $\Gamma,\Gamma'$ are simple closed curves as the proof of Proposition \ref{gen_lp} which we take to be non-intersecting (assume without loss of generality that $\Gamma'$ contains $\Gamma$ in the interior region). Since $$\lim_{p\rightarrow{\infty}}F_{p}(z_1,z_2)=F(z_1,z_2)=-\frac{1}{\gamma z_1z_2\underline{m}(z_1)\underline{m}(z_2)}+\frac{\underline{m(z_2)}-\underline{m}(z_1)}{\gamma z_1z_2\underline{m}(z_1)^2\underline{m}(z_2)^2(z_2-z_1)},$$ the same reasoning as in Proposition \ref{gen_lp} gives $$\frac{1}{p}tr\left(\Sigma h\left(S_n\right)\Sigma h\left(S_n\right)\right)\xrightarrow{a.s.}\left(\frac{1}{2\pi i}\right)^2\oint_{\Gamma}\oint_{\Gamma'}h(z_1)h(z_2)F(z_1,z_2)dz_1dz_2.$$

%First of all, we see that shrinking the curve $\Gamma'$ to approach the real axis gives: $$\frac{1}{2\pi i}\oint_{\Gamma'}h(z_1)F(z_1,z_2)dz_1=\lim_{\epsilon\downarrow{0}}\frac{1}{2\pi i}\oint_{\Gamma'}(F(x-i\epsilon,z_2)-F(x+i\epsilon,z_2))dx$$
%$$+h(0)I_{\gamma>1}\left(-\frac{1}{\gamma z_2\underline{m}(z_2)\underline{m}(0)}+\frac{\underline{m}(z_2)-\underline{m}(0)}{\gamma z_2^2\underline{m}(0)^2\underline{m}(z_2)^2}\right)$$

%$$=-\int h(x)\frac{g(x)}{\pi \gamma %x(f(x)^2+g(x)^2) z_2\underline{m}(z_2)}dx$$
%$$+\int h(x)\frac{2f(x)g(x)\underline{m}(z_2)-g(x)(f(x)^2+g(x)^2)}{\gamma\pi x z_2(z_2-x)\underline{m}(z_2)^2(f(x)^2+g(x)^2)^2} dx$$
%$$+h(0)I_{\gamma>1}\left(-\frac{1}{\gamma z_2\underline{m}(z_2)\underline{m}(0)}+\frac{\underline{m}(z_2)-\underline{m}(0)}{\gamma z_2^2\underline{m}(0)^2\underline{m}(z_2)^2}\right).$$

%Now, if in turn we shrink the curve $\Gamma$ to approach the real axis gives  $$\left(\frac{1}{2\pi i}\right)^2\oint_{\Gamma}\oint_{\Gamma'}h(z_1)h(z_2)F_{p}(z_1,z_2)dz_1dz_2$$
We can rewrite 

\begin{equation}\begin{split}
\left(\frac{1}{2\pi i}\right)^2\oint_{\Gamma}\oint_{\Gamma'}h(z_1)h(z_2)F(z_1,z_2)dz_1dz_2\\=\left(\frac{1}{2\pi i}\right)^2\oint_{\Gamma}\oint_{\Gamma'}h(z_1)h(z_2)\overline{F}(z_1,z_2)dz_1dz_2-\frac{1}{\gamma}\left(\frac{1}{2\pi i}\oint_{\Gamma} \frac{h(z)}{z\underline{m}(z)}dz\right)^2,
\end{split}\end{equation} where $$\overline{F}(z_1,z_2)=\frac{\underline{m}(z_2)-\underline{m}(z_1)}{\gamma z_1z_2\underline{m}(z_1)^2\underline{m}(z_2)^2(z_2-z_1)}.$$

By the calculation in the first part we know that  \begin{equation}\label{TERM2}\begin{split}\frac{1}{2\pi i}\oint_{\Gamma}-\frac{h(z)}{z\underline{m}(z)}dz=\frac{1}{2\pi i}\oint _{\Gamma}h(z)\frac{1+z\underline{m}(z)}{(-z\underline{m}(z))}dz\\=\int \frac{h(x)}{ \pi x (f(x)^2+g(x)^2)}g(x)dx+I_{\gamma>1}\frac{h(0)}{\underline{m}(0)}.
\end{split}\end{equation}

%Let us study the second term which arises from $\gamma>1$ first. When we evaluate $\displaystyle\frac{1}{2\pi i}\oint_{\Gamma} h(z_2)A(z_2)dz_2$ by shrinking $\Gamma$, the second summand of \eqref{kern} will give a term equal to $\frac{h(0)^2\underline{m}'(0)}{\gamma \underline{m}(0)^4}$ because of the contribution of the second order pole at $0$. From the rest of the support part of the contribution (if we ignore the singularity of $\frac{1}{z_2-x}$ for the moment), as we can see by shrinking $\Gamma$, equals \begin{equation}\label{conr}\begin{split}
%h(0)\int \frac{(f(y)^2+g(y)^2)g(y)-2f(y)g(y)\underline{m}(0)}{\pi \gamma y^2\underline{m}(0)^2(f(y)+g(y))^2}dyI_{\gamma>1}\\+\lim_{\epsilon\downarrow{0}}\int\int  h(x)h(y)\frac{1}{\pi} Im\left(\frac{2\underline{m}(y-i\epsilon)f(x)g(x)-g(x)\left(f(x)^2+g(x)^2\right)}{\gamma\pi xy(y-x) \left(f(x)^2+g(x)^2\right)^2\underline{m}(y-i\epsilon)^2 }\right) dxdy\\=h(0)\int \frac{(f(y)^2+g(y)^2)g(y)-2f(y)g(y)\underline{m}(0)}{\pi \gamma y^2\underline{m}(0)^2(f(y)+g(y))^2}dyI_{\gamma>1}+\int \int \overline{K}(x,y)h(x)h(y)dxdy
%\end{split}\end{equation}

%Here, a simple calculation shows that $$\overline{K}(x,y)=K(x,y)+\frac{g(x)g(y)}{\gamma\pi^2xy(f(x)^2+g(x)^2)(f(y)^2+g(y)^2)}.$$

%To sum up, the contribution until now is given by \eqref{TERM2},\eqref{conr} and equals $$\int \int K(x,y)h(x)h(y)$$

As a consequence, to prove Proposition \ref{Prop_LDA} it is enough to prove that \begin{equation}\label{left_to_prove}\begin{split}\left(\frac{1}{2\pi i}\right)^2\oint_{\Gamma}\oint_{\Gamma'}h(z_1)h(z_2)\overline{F}(z_1,z_2)dz_1dz_2\\=\int \int \overline{K}(x,y)h(x)h(y)dxdy+\int \frac{h(x)^2g(x)}{\gamma\pi x^2(f(x)^2+g(x)^2)^2}dx\\+\left[\frac{\underline{m}'(0)}{\gamma\underline{m}(0)^4}h(0)^2+\frac{2h(0)}{\gamma\underline{m}(0)^2}\int h(x)\frac{f(x)^2+g(x)^2-2f(x)\underline{m}(0)}{\pi x^2(f(x)^2+g(x)^2)^2}g(x)dx\right]I_{\gamma>1}.\end{split}
\end{equation}

Here $$\overline{K}(x,y)=K(x,y)+\frac{g(x)g(y)}{\gamma\pi^2xy(f(x)^2+g(x)^2)(f(y)^2+g(y)^2)}.$$

Following the same idea as in the proof of Proposition \ref{gen_lp}, we can shrink the curve $\Gamma'$ down to the real axis to get \begin{equation}\label{kern}\begin{split}
A(z_2)=\frac{1}{2\pi i}\oint_{\Gamma'}h(z_1)\overline{F}(z_1,z_2)dz_1\\=\lim_{\epsilon\downarrow{0}}\int \frac{h(x)}{2\pi i}\left[\overline{F}(x-i\epsilon,z_2)-\overline{F}(x+i\epsilon,z_2)\right]dx+\frac{h(0)(\underline{m}(z_2)-\underline{m}(0))}{\gamma z_2^2\underline{m}(0)^2\underline{m}(z_2)^2}I_{\gamma>1}\\=\int \frac{2\underline{m}(z_2)f(x)g(x)-g(x)(f(x)^2+g(x)^2)}{\gamma\pi xz_2(z_2-x)(f(x)^2+g(x)^2)^2\underline{m}(z_2)^2}h(x)dx+\frac{h(0)(\underline{m}(z_2)-\underline{m}(0))}{\gamma z_2^2\underline{m}(0)^2\underline{m}(z_2)^2}I_{\gamma>1}.
\end{split}\end{equation}

The first term above comes from the support of $F_{\gamma,H}$ on $(0,\infty)$, while the second is from the mass at $0$ when $\gamma>1$. Let us focus on the contribution of the second term in the case $\gamma>1$. We in turn shrink $\Gamma$ to calculate the contribution of the second term \begin{equation}\label{gamma_term}\begin{split}
\frac{1}{2\pi i}\oint_{\Gamma} h(z_2)\frac{h(0)(\underline{m}(z_2)-\underline{m}(0))}{\gamma z_2^2\underline{m}(0)^2\underline{m}(z_2)^2}dz_2=h(0)^2\frac{\underline{m}'(0)}{\gamma \underline{m}(0)^4}\\+\frac{1}{\pi}\lim_{\epsilon\downarrow{0}}\int Im \left[h(x-i\epsilon)\frac{h(0)(\underline{m}(x-i\epsilon)-\underline{m}(0))}{\gamma(x-i\epsilon)^2\underline{m}(0)^2\underline{m}(x-i\epsilon)^2} \right]dx\\=h(0)^2\frac{\underline{m}'(0)}{\gamma \underline{m}(0)^4}+\frac{h(0)}{\gamma\underline{m}(0)^2}\int h(x)\frac{f(x)^2+g(x)^2-2f(x)\underline{m}(0)}{\pi x^2(f(x)^2+g(x)^2)^2}g(x) dx.
\end{split}\end{equation}

Comparing \eqref{left_to_prove} and \eqref{gamma_term} we conclude that it remains to prove that \begin{equation}\begin{split}
\frac{1}{2\pi i}\oint_{\Gamma}\int \frac{2\underline{m}(z_2)f(x)g(x)-g(x)(f(x)^2+g(x)^2)}{\gamma\pi xz_2(z_2-x)(f(x)^2+g(x)^2)^2\underline{m}(z_2)^2} h(z_2)h(x)dxdz_2\\=\int\int \overline{K}(x,y)h(x)h(y) dxdy+\int h(x)^2\frac{g(x)}{\gamma\pi x^2(f(x)^2+g(x)^2)^2}dx\\+\left[\frac{h(0)}{\gamma\underline{m}(0)^2}\int h(x)\frac{f(x)^2+g(x)^2-2f(x)\underline{m}(0)}{\pi x^2(f(x)^2+g(x)^2)^2}g(x) dx\right]I_{\gamma>1}\\.
\end{split}\end{equation}

Changing the order of integration we see that the second term comes by the contribution at $0$ of the integrand in the case $\gamma>1$ by the Cauchy integral formula. For the rest of the proof we focus on the contribution of the rest of the support of $F_{\gamma,H}.$

There are two types of terms we need to consider here. First of all, ignoring the singularity at $z_2=x$ of the integrand and shrinking $\Gamma$ as before gives a term equal to $$\int\int\frac{1}{\pi} Im\left[\frac{2\underline{m}(y-i\epsilon)f(x)g(x)-g(x)(f(x)^2+g(x)^2)}{\gamma\pi xy(y-x)(f(x)^2+g(x)^2)^2\underline{m}(y-i\epsilon)^2} \right] h(x)h(y) dx dy$$
$$=\int\int \overline{K}(x,y)h(x)h(y)dx dy,$$ as we see by a straightforward calculation.

The last thing we need to take into  account  is the contribution of the singularity of $\frac{1}{z_2-x}$ at $x.$ Let $$k(z_2;x)=\frac{2\underline{m}(z_2)f(x)g(x)-g(x)(f(x)^2+g(x)^2)}{\gamma\pi xz_2(f(x)^2+g(x)^2)^2\underline{m}(z_2)^2}.$$ An important observation is that $k(\overline{z_2};x)=\overline{k(z_2;x)}\forall z_2\in\mathbb{C}-\mathbb{R}.$ The contribution of the singularity at $x$ is equal to

\begin{equation}\begin{split}
\lim_{r\downarrow{0}}\frac{1}{2\pi i}\int_0^{\pi}\frac{k(x+re^{i\theta};x)}{re^{i\theta}}h(x)h(x+re^{i\theta}) rie^{i\theta}dr\\+\frac{1}{2\pi i}\int_{\pi}^{2\pi}\frac{k(x+re^{i\theta};x)}{re^{i\theta}}h(x)h(x+re^{i\theta})rie^{i\theta}dr
\\=\lim_{r\downarrow{0}}\frac{1}{2\pi }\int_0^{\pi}\left[k(x+re^{i\theta};x)+k(x-re^{i\theta};x)\right] h(x)^2dx.
\end{split}
\end{equation}

To finish the proof we use dominated convergence and the fact that $$\lim_{r\downarrow{0}}Re\left[k(x+re^{i\theta};x)\right]=\frac{g(x)\lim_{r\downarrow{0}}Re\left[\frac{2\underline{m}(x+re^{i\theta})f(x)-f(x)^2-g(x)^2}{\underline{m}(x+re^{i\theta})^2}\right]}{\gamma\pi x^2(f(x)^2+g(x)^2)^2}$$ and (for $\theta\in [0,\pi]$)
\begin{equation}\begin{split}
\lim_{r\downarrow{0}}Re\left[\frac{2\underline{m}(x+re^{i\theta})f(x)-f(x)^2-g(x)^2}{\underline{m}(x+re^{i\theta})^2}\right]\\=Re\left[\frac{2(f(x)+ig(x))f(x)-f(x)^2-g(x)^2}{(f(x)+ig(x))^2}\right]=1.
\end{split}
\end{equation}

As a consequence, $$\lim_{r\downarrow{0}}Re\left[k(x+re^{i\theta};x)\right]=\frac{g(x)}{\gamma\pi x^2(f(x)^2+g(x)^2)^2}.$$

\end{proof}

We are now ready to prove the main theorem of Section \ref{DiscriminantAnalysis}, namely Theorem  \ref{main_discr}. With the machinery we have developed above, the proof is going to be rather straightforward. Since all of the concentration of measure results for quadratic forms that we will need in the next proof have been used a great number of times throughout the paper, for instance in the proof of Theorem \ref{main_1} and Proposition \ref{modif_lda}, we will omit some details when we use them.

\begin{proof}[Proof of Theorem \ref{main_discr}]

Following the notation of Section \ref{DiscriminantAnalysis} we have $\displaystyle\hat{\delta}=\delta+\Sigma^{{1}/{2}}{u}/{\sqrt{n}},$ where $u\sim \mathcal{N}(0,I_p)$ is independent of the sample covariance matrix $\hat{\Sigma}$. If we have a new data point $x=\Sigma^{{1}/{2}}z+\delta$, then the point is missclassified, using a shrinkage function $h$, if $$\hat{\delta}^\intercal h\left(\hat{\Sigma}\right)\Sigma^{\frac{1}{2}}z\leq -\hat{\delta}^\intercal h\left(\hat{\Sigma}\right)\delta.$$ This happens with probability \begin{equation}\label{error_rate_LDA}\Phi\left(-\frac{\hat{\delta}^\intercal h\left(\hat{\Sigma}\right)\delta }{{\norm{\Sigma^{\frac{1}{2}}h\left(\hat{\Sigma}\right)\hat{\delta}}}}\right),\end{equation} which is also, by symmetry, the missclassification rate of Discriminant Analysis in general.

Notice that $$\frac{u^\intercal\Sigma^{\frac{1}{2}}}{\sqrt{n}}h\left(\hat{\Sigma}\right)\delta\xrightarrow{a.s.}0,$$ since $u,\delta$ are independent (for example, by a $\frac{(4+\eta)}{2}$-th moment argument and Lemma \ref{lp_lemma} as in the proof of Theorem \ref{main_1}). This implies that $$\hat{\delta}^\intercal h\left(\hat{\Sigma}\right) \delta-{\delta}^\intercal h\left(\hat{\Sigma}\right) \delta\xrightarrow{a.s.}0.$$

Before we continue we observe that because $\hat{\Sigma}$ is a rank 1 perturbation of $S_n$ and they are both independent of $\delta,$ we can replace $\hat{\Sigma}$ by $S_n$ in all the asymptotics that follow and $u$ by a standard normal random variable in $\mathbb{R}^n)$.

By Lemma \ref{lp_lemma} and Borel-Cantelli as before, the above convergence result implies that \begin{equation}\label{fin_term1}\hat{\delta}^\intercal h\left(S_n\right)\delta-\frac{\alpha^2}{p}tr\left(h\left(S_n\right)\right)\xrightarrow{a.s.}0\Rightarrow{\hat{\delta}^\intercal h\left(S_n\right)}\delta\xrightarrow{a.s.}\alpha^2\int h(x)dF_{\gamma,H}(x).\end{equation}

Now we turn our focus to the asymptotics of $$\norm{\Sigma^{\frac{1}{2}}h\left(S_n\right)\hat{\delta}}^2.$$ Expanding the square of the norm we have \begin{equation}\label{fin_term2}\begin{split}\norm{\Sigma^{\frac{1}{2}}h\left(S_n\right)\hat{\delta}}^2=\delta^\intercal h\left(S_n\right)\Sigma h\left(S_n\right)\delta+\frac{1}{n}u^\intercal \Sigma^{\frac{1}{2}}h\left(S_n\right)\Sigma h\left(S_n\right) \Sigma^{\frac{1}{2}}u\\+2\delta^\intercal h\left(S_n\right) \Sigma h\left(S_n\right) \frac{\Sigma^{\frac{1}{2}}u}{\sqrt{n}}. \end{split}\end{equation}

We have the following convergence results, all of which follow from Lemma \ref{lp_lemma} and Borel-Cantelli exactly in the same fashion as earlier:

\begin{equation}\label{norm_converg1}\begin{split}
\delta^\intercal h\left(S_n\right) \Sigma h\left(S_n\right) \frac{\Sigma^{\frac{1}{2}}u}{\sqrt{n}}\xrightarrow{a.s.}0
\end{split}
\end{equation}

\begin{equation}\label{norm_converg2}\begin{split}
\delta^\intercal h\left(S_n\right)\Sigma h\left(S_n\right)\delta -\frac{\alpha^2}{p}tr\left(\Sigma h\left(S_n\right)^2\right)\xrightarrow{a.s.}0\\ \Rightarrow{\delta^\intercal h\left(S_n\right)\Sigma h\left(S_n\right)\delta\xrightarrow{a.s.}\alpha^2M_{\gamma,H}(h^2)}
\end{split}
\end{equation}

\begin{equation}\label{norm_converg3}\begin{split}
\frac{1}{n}u^\intercal \Sigma^{\frac{1}{2}}h\left(S_n\right)\Sigma h\left(S_n\right) \Sigma^{\frac{1}{2}}u-\frac{1}{n}tr\left(\Sigma h\left(S_n\right)\Sigma h\left(S_n\right)\right)\xrightarrow{a.s.}0\\\Rightarrow{\frac{1}{n}u^\intercal \Sigma^{\frac{1}{2}}h\left(S_n\right)\Sigma h\left(S_n\right) \Sigma^{\frac{1}{2}}u}\xrightarrow{a.s.}\gamma T_{\gamma,H}(h)
\end{split}\end{equation}

Using \eqref{norm_converg1},\eqref{norm_converg2},\eqref{norm_converg3}, we conclude from \eqref{fin_term2} that $$\norm{\Sigma^{\frac{1}{2}}h\left(S_n\right)\hat{\delta}}^2\xrightarrow{a.s.}\gamma T_{\gamma,H}(h)+\alpha^2M_{\gamma,H}(h^2).$$

The last result together with \eqref{fin_term1} and \eqref{error_rate_LDA} imply Theorem \ref{main_discr}.

\end{proof}

\bibliographystyle{plainnat}
\bibliography{regular}

\end{document}